\documentclass{article}

\usepackage[latin1, utf8]{inputenc}

\usepackage[T1]{fontenc}

\usepackage[francais,english]{babel}

\usepackage{amsfonts}
\usepackage{mathrsfs} 
\usepackage{lmodern}
\usepackage{titlesec}
\usepackage{graphicx}
\usepackage{hyperref}
\usepackage{amsmath}

\usepackage{amsthm}
\usepackage[left=2.5cm,top=3cm,right=2.5cm,bottom=3cm,bindingoffset=0.5cm]{geometry}
\titleformat{\chapter}[hang]{\bf\huge}{\thechapter}{2pc}{}
\newcommand{\RR}{\mathbb{R}}
\newcommand{\NN}{\mathbb{N}}
\newcommand{\e}{\varepsilon}
\newcommand{\wave}{\square}
\newcommand{\ce}{\mathit{C}}

\newcommand{\De}{\mathscr{D}}

\newtheorem{Prop}{Proposition}[section]

\newtheorem{Th}{Theorem}[section]

\newtheorem{Def}{Definition}[section]

\newtheorem{Rq}{Remark}[section]
 
\numberwithin{equation}{section}
\numberwithin{Th}{section}
\numberwithin{Prop}{section}
\numberwithin{Lm}{section}
\everymath{\displaystyle}

\title{Internal control of systems of semilinear coupled 1-D wave equations}
\author{Christophe ZHANG}
\date{}
\begin{document}
\maketitle
\selectlanguage{english}
\begin{abstract}
We prove the internal controllability of some systems of two coupled wave equations in one space dimension, with one control, under certain conditions on the coupling. To do this we apply the ``fictitious control method'' in two cases: general systems with a ``non-degenerate'' coupling, and a particular case where the coupling is ``degenerate'', namely a cubic coupling. 

In the latter case, our proof requires to find nontrivial trajectories of the control system that go from $0$ to $0$. We build these trajectories by adapting (in $1$ space dimension) a construction developed by Jean-Michel Coron, Sergio Guerrero and Lionel Rosier for the study of coupled parabolic systems.

\end{abstract}

\section{Main results and outline of proof}

\subsection{Control systems}\label{ControlSystem}

\

Let $T>0$, and $0< a < b < L$. We study the following class of systems: 
\begin{equation}\label{SystemGeneral}\left\{\begin{aligned}
u_{tt} - \nu_1^2 u_{xx} & = f_1(u,v)+ h, \ x \in [0,L], \\
 v_{tt} - \nu_2^2 v_{xx}  & = f_2(u,v), \ x \in [0,L], \\
u & = 0 \ \textrm{on} \ \{0,L\}, \\
v & = 0 \ \textrm{on} \ \{0,L\},
\end{aligned} \right.\end{equation}
where $h$ is the control, $f_1, f_2 \in \ce^\infty(\RR^2)$, $f_1(0,0)=f_2(0,0)=0$, $\nu_1, \nu_2\neq 0$. In what follows we shall note, for any such $\nu \neq 0$,
$$
\square_\nu:=\partial_{tt} - \nu^2 \partial_{xx}
$$
We will also study the following particular system:
\begin{equation}\label{System}\left\{\begin{aligned}
\square_{\nu_1} u & =  h, \ x \in [0,L], \\
\square_{\nu_2} v & =  u^3, \ x \in [0,L], \\
u & = 0 \ \textrm{on} \ \{0,L\}, \\
v & = 0 \ \textrm{on} \ \{0,L\}.
\end{aligned} \right.\end{equation}

These are systems of coupled semilinear wave equations, with different speeds, which we seek to control with a single control, which takes the form of a source term in the first equation. In both cases, as we will study solutions with $\ce^k$ regularity in order to establish a controllability result with two controls, the initial and final conditions $((u_0,u_1), (v_0, v_1), (u_0^f, u_1^f), (v_0^f,v_1^f))$ have to satisfy some compatibility conditions. For example, the conditions of order $1$ and $2$ read as:
\begin{equation}\label{2Compatibility}
\forall  \beta  \in \{0,L\} ,
\left\{\begin{aligned} u_0(\beta)&=u_1(\beta)=(u_0^f)(\beta)=(u_1^f)(\beta)=0, \\ 
u_0^{\prime\prime}(\beta)&=u_1^{\prime\prime}(\beta)=(u_0^f)^{\prime\prime}(\beta)=(u_1^f)^{\prime\prime}(\beta)=0, \\
  v_0(\beta)&=v_1(\beta)=(v_0^f)(\beta)=(v_1^f)(\beta)=0, \\
  v_0^{\prime\prime}(\beta)&=v_1^{\prime\prime}(\beta)=(v_0^f)^{\prime\prime}(\beta)=(v_1^f)^{\prime\prime}(\beta)=0.  
 \end{aligned}\right.
\end{equation}

To write the compatibility conditions of order $k\geq 3$, the idea is to first write the time derivatives of $u$ and $v$ as a function of their lower order space and time derivatives. 

There exists a multivariate polynomial $Q_{n,i}^{f_i}$ such that
\begin{equation}\label{TimePolynomial}
\left(\frac{d}{dt}\right)^n \left(f_i(u,v)\right)= Q_{n,i}^{f_i}\left(J_t^n(u,v)\right), \ i=1,2, 
\end{equation}
where $J_t^n(u,v)$ denotes the $n$-jet of time derivatives of $u$ and $v$, that is
$$\left(u, v, u_t, v_t,\cdots,\partial_t^n u,\partial_t^n v\right).$$

Now, define by recurrence the following family of operators:
\begin{equation}\label{CompatibilityOperators}
\left\{\begin{aligned}
\mathcal{D}_1^{(i)}&=\partial_t \\
\mathcal{D}_2^{(i)}&=\partial_{xx}+f_i(\cdot,\cdot), \\
 \mathcal{D}_n^{(i)}&=\partial_{xx} \circ \mathcal{D}_{n-2}^{(i)} + Q_{n-2,i}^{f_i}\left(J_t^{n-2}(\cdot,\cdot)\right), \ \textup{for} \ 3\leq n \leq k.
\end{aligned}\right.
\end{equation}
Then, near the corners $\Gamma:=\{(0,0), (0,L), (T,0), (T,L)\}$, using the equations of system \eqref{SystemGeneral} and keeping in mind that the control $h$ is supported away from the corners, we have
\begin{equation}\label{LocalComp}
\left\{\begin{aligned}
\partial_{t}^n u&=\mathcal{D}_n^{(1)} (u,v) \\
\partial_{t}^n v&=\mathcal{D}_n^{(2)} (u,v)
\end{aligned}\right.
\end{equation}

Now, thanks to the boundary conditions,
$$\partial_t^n u (c)=\partial_t^n v (c)=0, \ \forall c \in \Gamma, \forall n \leq k.$$
Moreover, it is clear thanks to the recurrence in \eqref{CompatibilityOperators} that there exist multivariate polynomials $P_{n,i}^{f_i}$ such that:
\begin{equation}\label{CompOpFinalForm}
\mathcal{D}_n^{(i)} (u,v) =P_{n,i}^{f_i}\left(J_x^n(u,v), J_x^{n-1}(u_t,v_t), J_t^n(u,v)\right), \forall n\leq k, i=1,2,
\end{equation}
where $J_x^n(u,v)$ denotes the $n$-jet of space derivatives f $u$ and $v$. Now, \eqref{LocalComp} can be written in the corners using only $u_0, u_1, u_0^f, u_1^f, v_0, v_1, v_0^f, v_1^f$, which gives the following compatibility conditions of order $k$:
\begin{equation}\label{CoupleCompatibility}
\left\{ \begin{aligned}
P_{n,i}^{f_i}\left(J_x^n(u_0,v_0)(0), J_x^{n-1}(u_1,v_1)(0), (0,\cdots,0)\right)&=0, \\
P_{n,i}^{f_i}\left(J_x^n(u_0,v_0)(L), J_x^{n-1}(u_1,v_1)(L), (0,\cdots,0)\right)&=0, \\
P_{n,i}^{f_i}\left(J_x^n(u_0^f,v_0^f)(0), J_x^{n-1}(u_1^f,v_1^f)(0),  (0,\cdots,0)\right)&=0,\\
P_{n,i}^{f_i}\left(J_x^n(u_0^f,v_0^f)(L), J_x^{n-1}(u_1^f,v_1^f)(L), (0,\cdots,0)\right)&=0,
\end{aligned}\right.
\quad \forall n\leq k, i=1,2.
\end{equation}
The existence and unicity of solutions to these systems can be derived from  TaTsien Li's general results on quasilinear wave equations (see \cite{LiRao} or \cite[chapter 5, section 5.2]{Li}).

The method we present yields two internal controllability results. The first is a local result:
\begin{Th}\label{NonDegenerateResult}
Let $R>0$, and $0\leq a < b \leq L$, $T>0$ such that
\begin{equation}\label{TimeCondTh1}
T > 2(L-b)\max\left(\frac{1}{\nu_1}, \frac{1}{\nu_2}\right), \ T>2a\max\left(\frac{1}{\nu_1}, \frac{1}{\nu_2}\right).\end{equation}
If 
\begin{equation}\label{NonDegenerateCoupling}
\frac{\partial f_2}{\partial u} (0,0) \neq 0, \end{equation}
then there exists $\eta >0$ such that for initial and final conditions 
$$((u_0,u_1),  (v_0, v_1), (u_0^{f}, u_1^{f}),(v_0^f, v_1^f) ) \in \left(B_{\ce^{11}([0,L])} \times B_{\ce^{10}( [0,L])}(0, \eta)\right)^4$$
where $B_{\ce^k}(0,\eta)$ denotes the ball centered in $0$ and with radius $\eta$ in the usual $\ce^k$ topology, satisfying \eqref{CoupleCompatibility} at the order $11$,
there exists $h \in \ce^6 ([0,T]\times[0,L])$ such that
\begin{equation}\label{SupportsTh1}
\textup{supp} \ h \subset [0, T] \times [a, b],\end{equation}
and such that the corresponding solution $(u,v)\in \ce^6([0,T]\times[0,L])^6$ of \eqref{SystemGeneral} with initial values $((u_0,u_1), (v_0, v_1))$ satisfies
$$\left\{\begin{aligned}u(T,\ \cdot \ )&=u_0^{f},&
u_t(T,\ \cdot \ )&=u_1^{f},\\
v(T, \ \cdot \ )&=v_0^f,&
v_t(T, \ \cdot \ )&=v_1^f\end{aligned}\right.$$
and 
\begin{equation}
\label{NonDegenerateEstimate}
\|(u,v,h)\|_{(\ce^6)^3} \leq  R.
\end{equation}

\end{Th}

The second theorem concerns a system that does not satisfy \eqref{NonDegenerateCoupling}. However, it is global, thanks to the system's homogeneity. 
\begin{Th}\label{GlobalResult}
Let $0\leq a < b \leq L$, $T>0$ satisfying \eqref{TimeCondTh1}.
There exists a constant $C>0$ depending on $T$ such that, for any given initial and final conditions 
$$((u_0,u_1),  (v_0, v_1), (u_0^{f}, u_1^{f}),(v_0^f, v_1^f) ) \in \left(\ce^{11}([0,L]) \times \ce^{10}( [0,L])\right)^4$$
satisfying \eqref{CoupleCompatibility} at the order 11,
there exists $h \in \ce^6 ([0,T]\times[0,L])$ such that
\begin{equation}\label{SupportsTh2}
\textup{supp} \ h \subset [0, T] \times [a, b],\end{equation}
and such that the corresponding solution $(u,v)\in \ce^6([0,T]\times[0,L])^2$ of \eqref{System} with initial values $((u_0,u_1), (v_0, v_1))$ satisfies
$$\left\{\begin{aligned}u(T,\ \cdot \ )&=u_0^{f},&
u_t(T,\ \cdot \ )&=u_1^{f},\\
v(T, \ \cdot \ )&=v_0^f,&
v_t(T, \ \cdot \ )&=v_1^f\end{aligned}\right.$$
and
\begin{equation}
\label{DegenerateEstimate}
\|h\|_{\ce^6} \leq C\left(\|(u_0,u_1,u_0^f,u_1^f)\|_{\left(\ce^{11}\times\ce^{10}\right)^2}+\|(v_0,v_1,v_0^f,v_1^f)\|^{\frac13}_{\left( \ce^{11}\times\ce^{10}\right)^2}\right).
\end{equation}

\end{Th}
\subsection{Related results} \label{Literature}

\

Boundary controllability results for quasilinear first order hyperbolic systems, coupled or not, can be found in Tatsien Li's book (\cite[chapter 3]{Li}).

As for second order systems, results of global boundary and internal controllability for the semilinear wave equation are well known, and were first proven by Enrique Zuazua in \cite{Z1} and \cite{Z2}. These articles introduced the use of HUM (Hilbert Uniqueness Method) to prove controllability results for semilinear and quasilinear equations. Boundary controllability results for scalar systems with $\ce^2$ regularity can be found in \cite[chapter 5]{Li}, and can be adapted to coupled systems, and for $\ce^k$ regularity.

Regarding controllability with a reduced number of controls, results for boundary and internal control of linear wave systems with a reduced number of controls have been proved by Fatiha Alabau-Boussouira (\cite{Al13} and \cite{Al2013}) in any space dimension, using energy methods. This was used by Fatiha Alabau Boussouira in \cite{Al14} to prove the existence of insensitizing controls for a single wave equation, as this is linked to the controllability of linear cascade systems in one space dimension, with the same speed in both equations. 

In the nonlinear case, Louis Tebou followed the same path for semilinear equations in \cite{Tebou}, where he proves the controllability of cascade systems of the form:
\begin{equation}\label{Tebou}
\left\{ \begin{aligned}
\wave u + f(u) & =  h+ \xi, \\
\wave v + f'(u)v &= 0, \\
u =0 ,
v( t, 0) & = \frac{\partial u}{\partial n} \chi_{\Gamma_0} \ \textrm{on} \  \partial \Omega,
\end{aligned}\right.
\end{equation}
where $\Gamma_0$ is a portion of the boundary, and where $f$ is subject to a growth constraint to have global well-posedness. To prove the controllability of such systems, the author first establishes the controllability of a linear problem, using a form of HUM combined with Carleman estimates. Then, using the Schauder fixed-point theorem, he establishes the controllability of the nonlinear problem. 

A similar strategy of proof appears in \cite{CGR} for parabolic systems with cubic coupling. In this case, as for system \eqref{System}, the linearised system around the equilibrium $(0,0,0)$ is not controllable. A classical tool to handle this problem in finite dimension is the use of iterated Lie brackets, see for example \cite[chapter 2]{Isi}, \cite[chapter 3]{NV}, and \cite[chapter 3]{Book}. However, this tool does not work (see for example \cite[chapter 5]{Book}) for many partial differential equations, including our control system \eqref{System}. In that case, a method to handle this situation is the return method. It consists in looking for trajectories going from $0$ to $0$ and such that the linearised system around them is controllable (return trajectories). This method has been introduced in \cite{CoronMCSS} for the stabilisation of driftless control systems and in \cite{CoronJMPA} and in \cite{CoronCRAS} for the controllability of the Euler equations of incompressible fluids. Following this method, in \cite{CGR} the authors build return trajectories, using the structure of the coupling.
Then, using Carleman estimates, they prove the controllability of a family of related parabolic linear systems close to the return trajectory, from which they deduce null-controllability using Kakutani's fixed-point theorem. 

In other cases, a phenomenon of loss of derivatives can occur when working with $\ce^k$ regularities: this can be handled with an inversion theorem of the Nash-Moser type, with a stronger condition on the linearised system.  This was done in the case for quasilinear first order hyperbolic systems, which have been studied in \cite{ACO}, using the ``fictitious control method'', which we will explain in the following section. More precisely the result that has been obtained concerns systems of the form:
$$\left\{\begin{aligned}
u_t + \Lambda_1(u,v)(u,v) + f_1(u,v) &= h, \\
v_t + \Lambda_2(u,v)(u,v) + f_2(u,v) &= 0.
\end{aligned}\right.$$
with
\begin{equation}\label{LinearCondition} \frac{\partial f_2}{\partial u} (0,0) \neq 0.\end{equation}

\subsection{The fictitious control method}\label{ReturnMethod}

\

The fictitious control method was introduced in \cite{CoronMCSS} and \cite{BurgosGarcia}, and  successfully used in \cite{CoronLissy}, \cite{ACO} and\cite{CoronGuilleron}. The idea is to first prove a controllability result with two controls (the fictitious controls), then reduce the number of controls, using some sort of fixed-point theorem.  

In this article, we apply it to second order hyperbolic systems, which present the same problem of loss of derivatives as the systems in \cite{ACO}. This loss of derivatives is handled by using Gromov's notion of algebraic solvability, which allows for differential operators to be inverted in a special way under some condition (infinitesimal inversion) on their derivative. This yields local results around the equilibrium, but we will also work around other trajectories than the stationary trajectory at the equilibrium, in the spirit of the return method, paying close attention to the regularities involved. Indeed, condition \eqref{NonDegenerateCoupling} from Theorem \ref{NonDegenerateResult} is identical to condition \eqref{LinearCondition}, and is crucial to solving the system algebraically (see Proposition \ref{InfinitesimalInversion}). If, as in Theorem \ref{GlobalResult}, it is not satisfied, then, following the spirit of the return method, one can build trajectories of the system along which such a condition is verified, at least on some appropriate spatial domain.

\begin{Rq}In both cases, conditions \eqref{NonDegenerateCoupling} and \eqref{LinearCondition} appear as a sufficient condition on the coupling for internal controllability with a reduced number of controls. However there is no indication (except in trivial cases like the linearised system above) that this sort of condition is necessary.
\end{Rq}

We can thus sum up our strategy of proof in three steps:
\begin{enumerate}
\item When necessary, find smooth trajectories around which Theorem \ref{GromovInversion} can be used.
\item Prove a local controllability result with two controls (fictitious controls) around the return trajectory, using classical boundary control results.
\item Use Theorem \ref{GromovInversion} to reduce the number of controls to one. 
\end{enumerate}

This article is organised as follows: in section 2, we illustrate Gromov's ideas on a linear example, and then prove Theorem \ref{NonDegenerateResult}, which is a case where we do not need to find return trajectories. This will allow us to present how Gromov's ideas can be applied in a nonlinear setting. In section 3 we prove Theorem \ref{GlobalResult}. In this case we need to find return trajectories, and the application of Theorem \ref{GromovInversion} around those trajectories will require a more detailed knowledge of the supports of the return trajectory. Finally section 4 is devoted to possible improvements and further questions on this topic. 

\section{The non-degenerate case}

\

As mentioned in section \ref{ReturnMethod}, we build on the method presented in \cite{ACO}. One of the main ingredients of this method is the theory of differential operators, and the notion of algebraic solvability, which we briefly present in the subsection below. The use of algebraic solvability in the study of control systems first appears in \cite{CoronMCSS}, where it was used to prove the stabilisability of finite dimensional systems without drift with time-varying feedbacks. It was first used in the context of partial differential equations in \cite{CoronLissy} for the control of the Navier-Stokes equation.

But first let us give an informal explanation of our method in the case of a linear system: first we have to rewrite the control problem using differential operators. We note $\mathscr{D}$ the operator associated with the equation of our control problem. Then, the control problem, given initial and final conditions, consists in finding $(u,v)$ with those initial and final conditions, and a control $h$ such that
$$\mathscr{D}(u,v,h)=0.$$
This corresponds to an inversion problem, but with a twist: one has to find an inverse image with the right initial and final conditions. Now, using the solutions to forward- and backward-evolving Cauchy problems corresponding to the initial and final conditions, one can build functions $(u,v)$ with the right initial and final conditions. The nonlinear version of this is done at the beginning of subsection \ref{Correction}. In general, one can do this so that for some $\eta>0$,
$$(h_1, h_2):=\De(u,v,0)=0 , \forall t\notin [\eta, T-\eta].$$ 

\noindent Now suppose $\De$ is invertible. We can make the following computation, the nonlinear version of which is made in subsection \ref{Correction}: 
 $$\De\left((u,v,0)+\De^{-1}(-h_1,-h_2)\right)=(h_1,h_2)-(h_1,h_2)=0.$$
 This seems to yield a solution to the control problem, however we still need to check that the ``corrective term'' does not change the initial and final conditions. This is where Gromov's notion of algebraic solvability comes into play: the right property for $\De$ is not to be invertible, but to be algebraically solvable. That is, that the inverse can also be written as a differential operator:
 $$\De^{-1}(-h_1,-h_2)=\sum_r a_r \partial_r (-h_1) + \sum_r b_r \partial_r (-h_2)$$
 for some functions $a_r, b_r$. With this additional property, one can see that, because $-h_1, -h_2$ vanish for $t\notin [\eta, T-\eta]$, 
 $$\De^{-1}(-h_1,-h_2)=0 , \forall t\notin [\eta, T-\eta].$$ 
 Hence, $(u,v,0)+\De^{-1}(-h_1,-h_2)$ still has the right initial and final conditions. 

\subsection{Differential relations and Gromov's theorem}

\ 

In this section we sum up some basic notions regarding differential operators, and Gromov's local inversion theorem for differential operators. More details can be found in \cite{Gr86}.

In what follows, $\mathcal{Q}$ is the closure of a non-empty open bounded smooth subset of $\RR^2$, and $p,q,r \in \NN^\ast$. We note $n_{r,p} := 2 + p \textup{ card} \{(\alpha_1, \alpha_2) \in \NN^2 \ | \ \alpha_1+\alpha_2 \leq r\}$. Recall the definition of the $r$-jet of a function $z\in \ce^r(\mathcal{Q})^p$:
$$J^r z (t,x)=\left( (t,x),z(t,x), \cdots , \frac{\partial^{|\alpha|} z}{\partial t^{\alpha_1} \partial x^{\alpha_2}}, \cdots, \frac{\partial^r z}{\partial t^{\alpha_1} \partial x^{\alpha_2}}\right) \in \RR^{n_{r,p}}, \ \forall (t,x) \in \mathcal{Q}.$$

\begin{Def}
A map $\mathscr{D} : \ce^r(\mathcal{Q})^p \rightarrow \ce^0(\mathcal{Q})^q$ is a $\ce^\infty$ nonlinear differential operator of order $r$ if there exists $F\in \ce^\infty(\RR^{n_{r,p}},\RR^q)$ such that
$$\mathscr{D}(z)=F(J^r z), \ \forall z \in\ce^r(\mathcal{Q})^p.$$
This clearly implies that $\mathscr{D}$ is $\ce^\infty$ (with the usual $\ce^r, \ce^0$ topologies), and we denote by 
$$\mathscr{L}_z : \ce^r(\mathcal{Q})^p \rightarrow \ce^0(\mathcal{Q})^q$$ 
its Fréchet differential at $z\in \ce^r(\mathcal{Q})^p$. 
\end{Def}

We now define some sort of manifold, over which we can invert these operators:
\begin{Def}
A subset $\mathcal{A}$ of $\ce^d(\mathcal{Q})^p$ is a differential relation of order $d\in \NN$ if there exists $\mathcal{R}\subset \RR^{n_{d,p}}$ such that
$$\mathcal{A}=\{z\in \ce^d(\mathcal{Q})^p \ | \ \forall (t,x) \in \mathcal{Q}, J^d z (t,x) \in \mathcal{R}\}.$$
It is said to be open if $\mathcal{R}$ is an open subset of $\RR^{n_{d,p}}$.
For $k\in \NN$, we note
$$\mathcal{A}^{k} := \mathcal{A} \cap \ce^{k}(\mathcal{Q})^p$$
\end{Def}

For classical local inversion theorems, one needs the differential at one point to be invertible. Here the requirement is somewhat stronger: we need the  differential at any point to be invertible, with the extra property that the inverse of each differential is also a linear differential operator.
\begin{Def}
Let $\mathcal{A} \subset \ce^d(\mathcal{Q})^p$ be a differential relation of order $d$, and let $\mathscr{D}$ be a differential operator of order $r$. We say that $\mathscr{D}$ admits an infinitesimal inversion of order $s\in \NN$ over $\mathcal{A}$ if there exists a family of linear differential operators of order $s$
$$z\in \mathcal{A}, \mathscr{M}_z: \ce^s(\mathcal{Q})^q \rightarrow \ce^0(\mathcal{Q})^p,$$
such that:
\begin{enumerate}
\item For every $g\in \ce^s(\mathcal{Q})^q$, $z \mapsto \mathscr{M}_z(g)$ is a differential operator of order d (possibly nonlinear) and it is a $\ce^\infty$-differential operator in $(z,g)$.
\item (Algebraic solvability) For every $z\in \mathcal{A}^{d+r}$,
$$\mathscr{L}_z \circ \mathscr{M}_z = \textrm{Id}_{\ce^{r+s}(\mathcal{Q})}.$$
\end{enumerate}
\end{Def}

We can now state Gromov's inversion theorem (see \cite[Section 2.3.2, main theorem]{Gr86}):
\begin{Th}[Gromov]\label{GromovInversion}
Let $\mathcal{A} \subset \ce^d(\mathcal{Q})^p$ be a non-empty open differential relation  of order $d$, and let $\mathscr{D}$ be a differential operator of order $r$. Assume that $\mathscr{D}$ admits an infinitesimal inversion of order $s$ over $\mathcal{A}$. Let
\begin{equation}
\sigma_0 > \max (d,2r+s),
\end{equation}
\begin{equation}
\nu \in (0, \infty).
\end{equation}
Then, there exists a family of sets $\mathcal{B}_z \subset \ce^{\sigma_0+s}(\mathcal{Q})^q$ and a family of operators $\mathscr{D}_z^{-1}:\mathcal{B}_z \rightarrow \mathcal{A}$ where $z\in \mathcal{A}^{\sigma_0+r+s}$, such that:
\begin{enumerate}
\item (Neighbourhood property) For every $z\in \mathcal{A}^{\sigma_0+r+s}$, $0\in \mathcal{B}_z$ and
$$\mathcal{B}:=\bigcup_{z\in \mathcal{A}^{\sigma_0+r+s}} \{z\} \times \mathcal{B}_z$$
is an open subset of $\ce^{\sigma_0+r+s}(\mathcal{Q})^p \times \ce^{\sigma_0+s}(\mathcal{Q})^q$.
\item (Inversion property) 
\begin{equation}\mathscr{D}\left(\mathscr{D}_z^{-1}(g)\right)=\mathscr{D}(z)+g, \ \forall (z,g) \in \mathcal{B}.\end{equation}
\item (Normalisation property) 
\begin{equation}
\mathscr{D}_z^{-1}(0)=z, \ \forall z \in \mathcal{A}^{\sigma_0+r+s}. \end{equation}
\item (Regularity and continuity)
Let $\sigma_0 \leq \sigma_1 \leq \eta_1$, then for all $z\in \mathcal{A}^{\eta_1+r+s}$ and $g\in \mathcal{B}_z^{\sigma_1+s}:=\mathcal{B}_z \cap \ce^{\sigma_1+s}$,
\begin{equation}
\label{Regularity}
\mathscr{D}^{-1}_z(g) \in \mathcal{A}^k, \quad \forall k < \sigma_1.
\end{equation}
Moreover, 
\begin{equation}\label{Continuity}
(z,g)\mapsto \mathscr{D}^{-1}_z(g) \in \ce^0(\mathcal{A}^{\sigma_0+r+s} \times \mathcal{B}_z^{\sigma_1+s}, \mathcal{A}^k), \quad \forall k< \sigma_1.
\end{equation}
Finally, if $\eta_1 > \sigma_1$, then \eqref{Regularity} and \eqref{Continuity} hold for $k=\sigma_1$.
\item (Locality) For every $(t,x) \in \mathcal{Q}$, and for every $(z_1,g_1), (z_2,g_2) \in \mathcal{B}$, if we have
$$(z_1,g_1)(\tilde{t}, \tilde{x}) =(z_2,g_2)(\tilde{t}, \tilde{x}), \ \forall (\tilde{t}, \tilde{x}) \in B((t,x),\nu) \cap \mathcal{Q},$$
then,
$$\mathscr{D}_{z_1}^{-1}(g_1)(t,x)=\mathscr{D}_{z_2}^{-1}(g_2)(t,x).$$
\end{enumerate}
\end{Th}

\begin{Rq}
\label{RqGromov}
The neighbourhood property allows to relate the domains of inversion for each local inversion to each other: local inverses at two ``neighbouring'' points will be defined on domains that have ``neighbouring'' sizes. In particular that means the domains of inversions are bound to overlap. The locality property tells us that when this happens (albeit locally), the images of the local inverses agree locally. In the linear case, this corresponds to the fact that when a function vanishes on an open set, its image by any linear differential operator also vanishes on this open set (see the beginning of the section).
\end{Rq}
\subsection{From two controls to one: algebraic solvability}\label{Correction}

As in the linear case, we first build a trajectory $(u,v)$ with the right initial and final conditions, but with $\De(u,v,0)$ potentially non-zero on some restricted domain. In terms of control theory, this amounts to solving the control problem with \textit{two} controls (the fictitious controls), with restricted supports. In fact, for systems of the form
\begin{equation}\label{SystemGeneral2controls}\left\{\begin{aligned}
\partial_{tt} u - \nu_1^2 u_{xx} & = f_1(u,v)+ h_1, \ x \in [0,L], \\
\partial_{tt} v - \nu_2^2 v_{xx}  & = f_2(u,v)+h_2, \ x \in [0,L], \\
u & = 0 \ \textrm{on} \ \{0,L\}, \\
v & = 0 \ \textrm{on} \ \{0,L\},
\end{aligned} \right.\end{equation}
where $f_1(0,0)=f_2(0,0)=0$, we have the following local controllability result, which is a consequence of boundary control results presented in \cite[chapter , sections 5.2 and 5.3]{Li}:
\begin{Prop} \label{InternalTwoControls1}
Let $k\geq 2$, $0 \leq a < b \leq L$, $T>0$ such that
\eqref{TimeCondTh1} holds.
For every $0<\delta <\min\left(T/2,(b-a)/2\right)$ satisfying 
\begin{equation}\label{delta}
\begin{aligned}
T-2\delta  & >  2(L-b+2\delta)\max\left(\frac{1}{\nu_1}, \frac{1}{\nu_2}\right), \\ 
T -2\delta & >  2(a+2\delta)\max\left(\frac{1}{\nu_1}, \frac{1}{\nu_2}\right),\end{aligned}\end{equation}
there exists $\eta >0$ such that, for initial and final conditions 
$$((u_0,u_1),  (v_0, v_1), (u_0^{f}, u_1^{f}),(v_0^f, v_1^f) ) \in \left(B_{\ce^{k}([0,T]\times [0,L])}(0,\eta) \times B_{\ce^{k-1}([0,T]\times [0,L])}(0,\eta)\right)^4$$
satisfying \eqref{CoupleCompatibility} at the order $k$,
there exist controls $h_1, h_2 \in \ce^{k-2}([0,T]\times [0,L])$ and constants $C_1, C_2>0$ depending on $T, \delta, k$ satisfying 
\begin{align}
\label{ControlSupports1}
&\textup{supp} \ h_i \subset [\delta, T-\delta]\times [a+\delta, b-\delta],
&i=1,2,\\
\label{hSmallness}
&\|h_i\|_{\ce^{k-2}} \leq C_1 \|((u_0,u_1),  (v_0, v_1), (u_0^{f}, u_1^{f}),(v_0^f, v_1^f) )\|_{\left(\ce^k \times \ce^{k-1}\right)^4},
&i=1,2,
\end{align}
such that the corresponding solution of \eqref{SystemGeneral2controls} with initial values $((u_0, u_1), (v_0, v_1))$ satisfies
$$\left\{\begin{aligned}u(T,\ \cdot \ )&=u_0^{f},&
u_t(T,\ \cdot \ )&=u_1^{f},\\
v(T, \ \cdot \ )&=v_0^f,&
v_t(T, \ \cdot \ )&=v_1^f.\end{aligned}\right.$$
\begin{equation}\label{ySmallness}
\|(u,v)\|_{(\ce^k)^2} \leq C_2 \|((u_0,u_1),  (v_0, v_1), (u_0^{f}, u_1^{f}),(v_0^f, v_1^f) )\|_{\left(\ce^k \times \ce^{k-1}\right)^4}.
\end{equation}
\end{Prop}

\

This result is a particular case of Proposition \ref{InternalTwoControls} which we will prove in the following section, when dealing with a degenerate system.

For now, let $R>0$, $0\leq a < b \leq L$, and let $T>0$ be such that \eqref{TimeCondTh1} holds. 
Let $0<\delta<\min\left(T/2,(b-a)/2\right)/2$ such that \eqref{delta} holds for $2\delta$ (note that it also holds for $\delta$). Define
$$\mathcal{Q}_\delta:=[\delta, T-\delta] \times [a+\delta, b-\delta],$$
$$\mathcal{Q}_{2\delta}:=[2\delta, T-2\delta] \times [a+2\delta, b-2\delta],$$
and let $\mathcal{Q}\subset [0,T]\times [a,b]$ be a smooth closed set such that
$$\mathcal{Q}_\delta \subset \overset{\circ}{\mathcal{Q}}.$$
Define the following nonempty open differential relation of order $2$:
$$\mathcal{A}=\left\{(u,v,h) \in \left(\ce^{2}(\mathcal{Q})\right)^3 \ \middle| \ \forall (t,x) \in \mathcal{Q}, \ \frac{\partial f_2}{\partial u} (u(t,x),v(t,x)) \neq 0 \right\}.$$
We define the following nonlinear differential operator $\mathscr{D}:\ce^2(\mathcal{Q})^3\rightarrow \ce^0(\mathcal{Q})^2$ of order $r=2$:
$$\mathscr{D}\left((u,v,h)\right)=(\wave_{\nu_1} u -f_1(u,v)-h, \wave_{\nu_2} v - f_2(u,v)), \ \forall (u,v,h) \in \ce^2(\mathcal{Q})^3,$$
and its differential at $(u,v,h) \in \ce^2([0,T]\times [0,L])^3$:
$$\mathscr{L}_{(u,v,h)}(\tilde{u},\tilde{v},\tilde{h})=\left(\wave_{\nu_1} \tilde{u} -Df_1 (u,v)\cdot (\tilde{u}, \tilde{v})-\tilde{h}, \ \wave_{\nu_2} \tilde{v} -Df_2 (u,v)\cdot (\tilde{u}, \tilde{v})\right), \ \forall (\tilde{u},\tilde{v},\tilde{h}) \in \ce^2([0,T]\times [0,L])^3.$$

We now have the following result, thanks to the definition of $\mathcal{A}$:
\begin{Prop}\label{InfinitesimalInversion}
$\mathscr{D}$ admits an infinitesimal inversion of order $2$ over $\mathcal{A}$.
\end{Prop}
\begin{proof}
Let $h_1, h_2 \in \ce^4(\mathcal{Q})$, $(u,v,h) \in \mathcal{A}$. Using the fact that $\frac{\partial f_2}{\partial u}(u,v)$ never vanishes, if we set:
$$\begin{aligned}
\tilde{v}&=0,\\
\tilde{u}&=\displaystyle{-\frac{h_2}{\frac{\partial f_2}{\partial u}(u,v)}},\\
\tilde{h}&=\wave_{\nu_1} \tilde{u} -\frac{\partial f_1}{\partial u}(u,v) \tilde{u}- h_1,
\end{aligned}$$
then we have
$$\mathscr{L}_{(u,v,h)}(\tilde{u},\tilde{v},\tilde{h})=(h_1,h_2).$$
Moreover, the above formulae clearly show that $(u,v,h) \mapsto \mathscr{L}_{(u,v,h)}(\tilde{u},\tilde{v},\tilde{h})$ is a (nonlinear, $\ce^\infty$ with the usual topology of $\ce^2(\mathcal{Q})$) differential operator of order $2$ on $\ce^2(\mathcal{Q})$, and $(u,v,h,\tilde{u},\tilde{v},\tilde{h})\mapsto \mathscr{L}_{(u,v,h)}(\tilde{u},\tilde{v},\tilde{h})$ is also $\ce^\infty$.
\end{proof}

\begin{figure}[h!]
 \centerline{\includegraphics[scale=0.8]{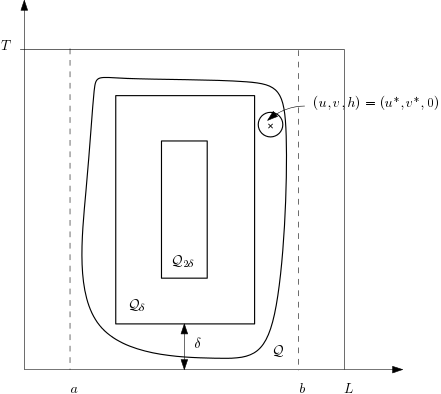}}
 		\caption{\label{fig:Gromov} \textit{Matching trajectories with two controls and with a single control on the appropriate domain.}}
 \end{figure}
We can now apply Theorem \ref{GromovInversion} with $d=2$, $s=2$, $r=2$ $\sigma_0=7$, $\nu=\delta/2$. This yields a collection of open sets, which all contain $0$,
$$\mathcal{B}_z \subset \left(\ce^{9}(\mathcal{Q})\right)^2, \ z \in \mathcal{A}^{11},$$
the open subset of $\left(\ce^{11}(\mathcal{Q})\right)^3 \times \left(\ce^{9}(\mathcal{Q})\right)^2$
$$\mathcal{B}=\bigcup_{z \in \mathcal{A}^{11}} \{z\} \times \mathcal{B}_z,$$
and the collection of operators
$$\mathscr{D}_z^{-1}:\mathcal{B}_z \rightarrow \mathcal{A}, \ z \in \mathcal{A}^{11}.$$
Now, thanks to condition \eqref{NonDegenerateCoupling},
$$(0,0,0) \in \mathcal{A},$$
$$\mathscr{D}(0,0,0)=(0,0),$$
and
$$\left((0,0,0), (0,0) \right) \in \mathcal{B},$$
so that, thanks to the neighbourhood property of Theorem \ref{GromovInversion}, there exists $\e>0$ such that 
\begin{equation}\label{InB} \left(B_{\ce^{11}(\mathcal{Q})}((0,0,0), \e)\right)^3 \times \left(B_{\ce^{9}(\mathcal{Q})}((0,0),\e)\right)^2 \subset \mathcal{B}.\end{equation}
By the continuity property of Theorem \ref{GromovInversion} with $\eta_1=\sigma_1=\sigma_0=7$, there exists $\eta>0$ such that for $\|((u,v,h),(h_1,h_2))\|_{(\ce^{11})^3 \times (\ce^9)^2} \leq \eta$,
$$\|\De^{-1}_{(u,v,h)}(h_1, h_2)\|_{(\ce^6)^3} \leq R.$$
Proposition \ref{InternalTwoControls1} with $k=11$ yields $\eta^\prime > 0$ such that for any initial and final conditions
$$((u_0,u_1),  (v_0, v_1), (u_0^{f}, u_1^{f}),(v_0^f, v_1^f) )\in \left(B_{\ce^{11}([0,T]\times [0,L])}(0,\eta^\prime) \times B_{\ce^{10}([0,T]\times [0,L])}(0,\eta^\prime)\right)^4,$$
there exist two controls $\theta_1, \theta_2 \in \ce^{9}([0,T]\times [0,L])$, supported in $\mathcal{Q}_{2\delta}$ (condition \eqref{ControlSupports1}), that steer system \eqref{SystemGeneral2controls} from the given initial conditions to the given final conditions, with the corresponding trajectory $(u^\ast,v^\ast)$ satisfying \eqref{ySmallness}.
Together with \eqref{InB}, this implies that there exists $\eta^\prime \geq \eta^{\prime\prime} >0$ such that for initial and final conditions
$$((u_0,u_1),  (v_0, v_1), (u_0^{f}, u_1^{f}),(v_0^f, v_1^f) )\in \left(B_{\ce^{11}([0,T]\times [0,L])}(0,\eta^{\prime\prime}) \times B_{\ce^{10}([0,T]\times [0,L])}(0,\eta^{\prime\prime})\right)^4,$$
the corresponding trajectory of system \eqref{SystemGeneral2controls} satisfies
\begin{equation}\label{image}
\mathscr{D}(u^\ast_{|\mathcal{Q}}, v^\ast_{|\mathcal{Q}}, 0)=\left(h_{1|\mathcal{Q}}, h_{2|\mathcal{Q}}\right),\end{equation}
\begin{equation}
\label{hiSmall}
\left((u^\ast_{|\mathcal{Q}}, v^\ast_{|\mathcal{Q}}, 0), (-h_{1|\mathcal{Q}}, -h_{2|\mathcal{Q}})\right) \in \mathcal{B}.\end{equation}
\begin{equation}\label{PreimageSmallness}
\|((u^\ast,v^\ast,0),(h_1,h_2))\|_{(\ce^{11})^3 \times (\ce^9)^2} \leq \min (R,\eta).
\end{equation}

Let us now set, keeping in mind the regularity property of Theorem \ref{GromovInversion} with $\eta_1=\sigma_1=\sigma_0=7$,
$$(u,v,h)=\mathscr{D}_{(u^\ast_{|\mathcal{Q}}, v^\ast_{|\mathcal{Q}}, 0)}^{-1}\left(-h_{1|\mathcal{Q}}, -h_{2|\mathcal{Q}}\right) \in \mathcal{A}^6.$$
Then, by the inversion property of Theorem \ref{GromovInversion}, and \eqref{image},
$$\mathscr{D}\left(u,v,h\right)=\mathscr{D}(u^\ast_{|\mathcal{Q}},v^\ast_{|\mathcal{Q}}, 0)-(h_{1|\mathcal{Q}}, h_{2|\mathcal{Q}})=(0,0).$$
Now, let us show that $(u,v,h)=(u^\ast, v^\ast,0)$ on $\mathring{\mathcal{Q}} \setminus \mathcal{Q}_{\delta^\prime}$. This will allow us to extend $(u,v,h)$ on  $\left([0,T]\times [0,L]\right) \setminus \mathcal{Q}$. 

Let $(t,x) \in \displaystyle{\mathring{\mathcal{Q}} \setminus \mathcal{Q}_{\delta}}$. As the $h_i$ are supported in $\mathcal{Q}_{2\delta}$, 
\begin{equation} \left((u^\ast, v^\ast, 0), (-h_1, -h_2)\right)=\left((u^\ast, v^\ast, 0), (0,0)\right) \ \ \textrm{on} \ \displaystyle{B\left((t,x), \frac\delta2 \right) \cap \mathcal{Q}.}\end{equation}
Thus, using the locality property of Theorem \ref{GromovInversion},
\begin{equation}\mathscr{D}_{(u^\ast_{|\mathcal{Q}}, v^\ast_{|\mathcal{Q}}, 0)}^{-1}\left(-h_{1|\mathcal{Q}}, -h_{2|\mathcal{Q}}\right)(t,x)=\mathscr{D}_{(u^\ast_{|\mathcal{Q}}, v^\ast_{|\mathcal{Q}}, 0)}^{-1}\left(0,0\right)(t,x),\end{equation}
that is, using the normalisation property:
\begin{equation}\label{Matching}
(u,v,h)(t,x)=(u^\ast,v^\ast, 0)(t,x).
\end{equation}

We can now extend $(u,v,h)$ by setting
\begin{equation}\ (u,v,h)(t,x)=(u^\ast,v^\ast, 0)(t,x), \ \forall (t,x) \in [0,T]\times [0,L] \setminus \mathcal{Q}.\end{equation}
Then, 
$$\textrm{supp} \ h \subset [0,T] \times [a, b],$$
and $(u,v)$ satisfies the same initial, boundary and final conditions as $(u^\ast,v^\ast)$:
 \begin{equation}
 \left\{
 \begin{aligned}
 (u,v)(0, \ \cdot \ ) &=(u_0, v_0), &
 (u_t,v_t)(0, \ \dot \ ) &= (u_1, v_1) \\
  (u,v)(T, \ \cdot \ ) &=(u_0^f, v_0^f),&
 (u_t,v_t)(T, \ \dot \ ) &= (u_1^f, v_1^f)
 \end{aligned}\right.\\
 \end{equation}
 \begin{equation}
 \left\{\begin{aligned}
 u(\ \cdot \ , 0) &= u( \ \cdot \ , L) =0 \\
 v(\ \cdot \ , 0) &= v( \ \cdot \ , L) =0 
 \end{aligned}\right.
 \end{equation}
 and  
\begin{equation}
\left\{ 
	\begin{aligned}
    \wave_{\nu_1} u &= f_1(u,v)+h, \\
    \wave_{\nu_2} v &= f_2(u,v),
    \end{aligned}
\right.
\end{equation}

Finally, we get \eqref{NonDegenerateEstimate} from \eqref{PreimageSmallness} and the continuity property of Theorem \ref{GromovInversion}.

\noindent This proves Theorem \ref{NonDegenerateResult}.

\begin{Rq}
Theorem \ref{NonDegenerateResult} actually holds for coupled quasilinear equations:
\begin{equation}\label{SystemQuasiLin}\left\{\begin{aligned}
\partial_{tt} u - \partial_x \left(K_1(u, \partial_x u)\right) & = f_1(u,v)+ h, \ x \in [0,L], \\
\partial_{tt} v - \partial_x \left(K_2(v, \partial_x v)\right)  & = f_2(u,v), \ x \in [0,L], \\
u & = 0 \ \textrm{on} \ \{0,L\}, \\
v & = 0 \ \textrm{on} \ \{0,L\},
\end{aligned} \right.\end{equation}
where $f_1(0,0)=f_2(0,0)=0$, $K_1, K_2 \in \ce^\infty(\RR^2)$, and $K_1(0,0)=K_2(0,0)=0$. One can check that when one modifies the recurrence relation in \eqref{CompatibilityOperators} to match the new equations, the operators can still be written using only $J_x^n(u,v)$, $J_x^{n-1}(u_t,v_t)$ and $J_t^n(u,v)$, and thus the compatibility conditions will have the same form as \eqref{CoupleCompatibility}.

Indeed, in this case we can still use Li's results for the perturbed quasilinear system, as we consider the ``perturbations'' around $0$. This will yield a ``universal'' time condition, because the propagation speeds are close to $\min (\sqrt{\partial_2 K_1(0,0)}, \sqrt{\partial_2 K_2(0,0)})$ for the perturbed system. On the other hand if we work around a nonzero trajectory (return method), the perturbed quasilinear system could present quite smaller propagation speeds. The final time condition would then depend on the return trajectories that are found.
\end{Rq}
\begin{Th}\label{QuasiLinResult}
Let $R>0$, $0\leq a < b \leq L$, $T>0$ such that
\begin{equation}\label{TimeCondThQuasiLin}
\begin{aligned}
T &> 2(L-b)\max \left(\left(\sqrt{\partial_2 K_1(0,0)}\right)^{-1}, \left(\sqrt{\partial_2 K_2(0,0)}\right)^{-1}\right), \\ T & >2a\max \left(\left(\sqrt{\partial_2 K_1(0,0)}\right)^{-1}, \left(\sqrt{\partial_2 K_2(0,0)}\right)^{-1}\right).\end{aligned}\end{equation}
If 
\begin{equation}\label{NonDegenerateCouplingQuasiLin}
\frac{\partial f_2}{\partial u} (0,0) \neq 0, \end{equation}
then there exists $\eta >0$ such that for initial and final conditions 
$$((u_0,u_1),  (v_0, v_1), (u_0^{f}, u_1^{f}),(v_0^f, v_1^f) )\in B_{\left(\ce^{11}([0,L]) \times \ce^{10}( [0,L])\right)^4}(0, \eta)$$
compatible at the order $11$,
there exists $h \in \ce^6 ([0,T]\times[0,L])$ such that
\begin{equation}\label{SupportsThQuasiLin}
\textup{supp} \ h \subset [0, T] \times [a, b],\end{equation}
and such that the corresponding solution $(u,v)\in \ce^6 ([0,T]\times[0,L])$ of \eqref{SystemQuasiLin} with initial values $((u_0,u_1), (v_0, v_1))$ satisfies
$$\left\{\begin{aligned}u(T,\ \cdot \ )&=u_0^{f},&
u_t(T,\ \cdot \ )&=u_1^{f},\\
v(T, \ \cdot \ )&=v_0^f,&
v_t(T, \ \cdot \ )&=v_1^f\end{aligned}\right.$$
and inequality \eqref{NonDegenerateEstimate} holds.

\end{Th}

\section{A degenerate case: cubic coupling}

\ 

We now turn to system \eqref{System}. Let us mention a few important specificities of this system. 
First, around the equilibrium $(0,0,0)$, the linearised system is obviously not controllable:
\begin{equation}\label{LinSystem}\left\{
\begin{aligned}
\wave_{\nu_1} u & = h, \\
\wave_{\nu_2} v & = 0, \\
u_{|\partial \Omega} & =  0, \\
v_{|\partial \Omega} & =  0,
\end{aligned} \right.
\end{equation}
the control $h$ gives us no influence on the dynamics of$v$.

This fact can also be described as a degenerescence: system \eqref{System} does not satisfy condition \eqref{NonDegenerateCoupling}, so the coupling can be seen as degenerate. Thus, the computations from the beginning of subsection \ref{Correction} do not hold: we cannot work around the stationary trajectory $0$, thus we need to find another trajectory around which to work. More precisely, keeping in mind Proposition \ref{InfinitesimalInversion}, we look for a return trajectory $(\bar{u}, \bar{v},\bar{h})$ going from $0$ to $0$ such that for some smooth closed set $\mathcal{Q} \subset [0,T]\times[a,b]$,we have
\begin{equation}\label{uSuitable}\forall (t,x) \in \mathcal{Q} ,\frac{\partial f_2}{\partial u}(\bar{u}(t,x), \bar{v}(t,x))=3\bar{u}^2 (t,x)\neq 0.\end{equation}
Additionnally, $\mathcal{Q}$ will have to satisfy some properties so that a result with two controls can be proved.

To find such a trajectory, we follow the same idea as in \cite{CGR}, where return trajectories are built for coupled heat equations with a cubic coupling. The additional derivative in time simply adds terms and makes for heavier computations. However, condition \eqref{uSuitable} will account for additional work.

We will then prove and use a more general controllability result with two controls. After that, the application of Gromov's theorem is rather straightforward.

\subsection{A preliminary construction: elementary trajectories}
\ 
 
In this subsection, we describe a construction of a smooth trajectory of system \eqref{System} that goes from $0$ to $0$. For now we consider condition \eqref{uSuitable} but without any special requirements for $\mathcal{Q}$.

In what follows, we suppose, without loss of generality (by scaling the space variable) that $\nu_2=1$.

To build trajectories that start at $0$ and return there, the idea is to use the cascade structure of the equation: first we find a $\ce^\infty([-1,1]\times [0,1])$ function $\bar{v}$ such that $\wave \bar{v}$ is the third power of a $\ce^\infty([-1,1]\times [0,1])$ function $\bar{u}$. By setting the right conditions at the start and end times, this gives us a return trajectory. The corresponding control will then be $\wave_{\nu_1} \bar{u}$.

Let us recall that $x\mapsto \sqrt[3]{x}$ is $\ce^\infty$ on $\RR^\ast$. So, by composition, the cubic root of a $\ce^\infty$ function $f$ is $\ce^\infty$ at all the points where $f$ is non-zero. At the points where $f$ vanishes, by Taylor's formula, a fairly simple sufficient condition for $\sqrt[3]{f}$ to be $\ce^\infty$ at those points is for $f$ to vanish, along with its first and second derivatives, while its third derivative is non-zero.

Now, to find functions whose image by the wave operator is a third power of a $\ce^\infty$ function, we consider the solutions to the corresponding stationary problem, namely functions whose Laplacian is the third power of a $\ce^\infty$ function. The solution of this problem corresponds to the following proposition, proven (with $1/2$ instead of $3/4$) in \cite{CGR}:
\begin{Prop}[Coron, Guerrero, Rosier] \label{Stationary1D}
There exist $\delta^\prime, \delta^{\prime\prime}, g\in \ce^\infty([0,1]), G \in \ce^\infty([0,1])$ such that
\begin{equation}\label{Stationary1DAppendix}
\left\{\begin{array}{l}
g'' = G,\\
g(z) = 1- z^2 \ \textup{on} \ [0, \delta^{\prime\prime}], \\
g(z) = e^{-\frac{1}{1-z^2}} \ \textup{on} \ [1-\delta^\prime, 1), \\
G(z)\left(z-\frac34\right)>0 \textup{ for } z\in (0,1) \setminus\left\{\frac34\right\},\\
G(z)=\left(z-\frac34\right)^3 \ \textup{on} \  \left[\frac34-\frac{\delta^{\prime\prime}}{2}, \frac34+\frac{\delta^{\prime\prime}}{2}\right],
\end{array}\right.\end{equation}

\end{Prop}

\ 

In a sense, this proposition gives us the simplest example of functions the second derivative of which is the third power of a smooth function: $G=g^{\prime\prime}$ vanishes exponentially in $1$, and has only one vanishing point on $[0, 1)$, around which it has a cubic behaviour.
 The idea of the construction is then to perturb this function of space and make it evolve in time, so slightly as to preserve the properties \ref{Stationary1DAppendix} of the stationary problem.
 Let $0\leq a < b \leq L$, and $T>0$ such that \eqref{TimeCondTh1} holds.

Let $0<\delta<\min(T/2, (b-a)/2)$ such that \eqref{lambda} holds. Set $\lambda_0$ to be a function such that
\begin{equation}\label{lambda}
\begin{array}{c}

\lambda_0(t)= e^{-\sqrt{\frac{1}{t(T-t)}}} \quad \forall t \in \left(0, \frac{\delta}{2}\right] \cup \left[T-\frac{\delta}{2}, T\right), \\
\lambda_0(0)=\lambda_0(T)=0, \\
\lambda_0(t)>0, \ \forall t \in (0,T), \\
\lambda_0([\delta, T-\delta])=\{1\},\end{array}\end{equation}
and write $\lambda:=\e \lambda_0$ for some $\e$ to be determined. 

\begin{Rq}\label{lambdaPol}
In \cite{CGR}, the authors take 
\begin{equation}\label{CGRLambda}\lambda(t)=\e t^2(1-t)^2.\end{equation}
In our case however, we will see that we need to fit a rectangle of the form $[\delta, T-\delta]\times[x_0-\xi, x_0+\xi]$ inside the support of $\bar{u}$, see Figure \ref{fig:Tuile}. With a polynomial as in \eqref{CGRLambda}, the smaller $\delta>0$ gets, the smaller $\xi$ has to be. This in itself would not be an obstruction to prove our controllability result, but using definition \eqref{lambda} has the advantage to fix the width of the rectangle for all $\delta$ satisfying \eqref{delta}.
\end{Rq}
\noindent Set
\begin{equation}\label{f_0}
\begin{array}{c}
f_0(t)=e^{- \frac{1}{t(t-T)}},\ \forall t \in (0,T),\\
f_0(0)=f_0(T)=0,
\end{array}
\end{equation}
Finally, let $g_0$ be the solution to the stationary problem (see Proposition \ref{Stationary1D}).
Let $x_0 \in (0,L)$, and choose $\e \leq \min (x_0, L-x_0)$. We now look for $\bar{v}$ in the form
\begin{equation}\label{LinCombi}\bar{v}(t,x)=\sum_{i=0}^3 f_i(t) g_i\left(\frac{|x-x_0|}{\lambda(t)}\right).\end{equation}
Note that the fact that $f_0$ vanishes faster than $\lambda$ at $0$ and $T$ compensates the singularity that occurs in the term $|x-x_0|/\lambda(t)$ of the first term of the sum. We will see that the $f_i$ have a similar property, thus ensuring that functions of the form above are indeed $\ce^{\infty}$.
We also require that the $g_i$ satisfy
\begin{equation}\label{supports}
\textrm{supp} \ g_i \subset \left[\frac{3}{4} -\frac{\delta^{\prime\prime}}{2}, \frac{3}{4} +\frac{\delta^{\prime\prime}}{2}\right], \ \forall i \in \{1,2,3\},\end{equation}
where $\delta^{\prime\prime}$ is as defined in Proposition \ref{Stationary1D}, so that
\begin{equation}
\textrm{supp} \ (\bar{u}, \bar{v}, \bar{h}) \subset [0,T] \times [x_0-\e,x_0+\e].
\end{equation}

 		\begin{figure}[h]
 		\centerline{\includegraphics[scale=0.8]{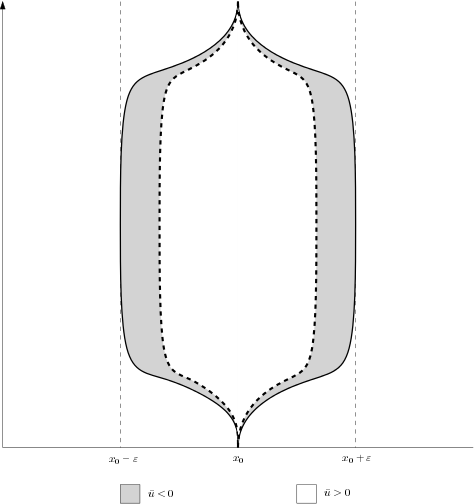}}
 		\caption{\label{fig:ElemSupport} \textit{The support of the trajectory $(\bar{u},\bar{v}, \bar{h})$. The dashed line represents the vanishing points of $\wave \bar{v}$ (or, equivalently, $\bar{u}$).}}
 		\end{figure}
Let us then set, in order to simplify the notations for our computations:
$$z:=\frac{|x-x_0|}{\lambda(t)},$$
$$V(t,x):=\wave \bar{v} = \bar{v}_{tt}-\bar{v}_{xx}, $$
which we note, in the new set of variables,
$$\mathcal{V}(t,z):=V(t, \lambda(t)z).$$

We are now looking for functions $f_i$ and $g_i$ such that $V^{\frac{1}{3}}$ is of class $\ce^\infty$. In order to achieve this, we will work with the new set of variables $(t,z)$, and study $\mathcal{V}$. We now need to have precise knowledge of the behaviour of $\mathcal{V}$ when it vanishes. 

\noindent More precisely, the aim is to write $\mathcal{V}$ near $\frac{3}{4} $ as:
$$\lambda^2 \mathcal{V}=\left(z-\frac{3}{4} \right)^3 \varphi(t,z), \ \textrm{with} \ \varphi \in \ce^\infty\left([0,T]\times \left[ \frac{3}{4} -\frac{\delta^{\prime\prime}}{2}, \frac{3}{4} +\frac{\delta^{\prime\prime}}{2}\right]\right), \ \varphi <0 \ \textrm{for}\ t\neq 0, T.$$
Note that $\varphi$ has to be negative because of the minus sign in the wave operator.
Hence, we look for $\mathcal{V}$ satisfying
$$\begin{array}{c}
\displaystyle{\mathcal{V}_z\left( \ \cdot \ , \frac{3}{4}  \right)=0}, \\
\displaystyle{\mathcal{V}_{zz}\left( \ \cdot \ , \frac{3}{4}  \right)=0},\\
\displaystyle{\mathcal{V}_{zzz} \leq -C f_0 } \ \textrm{on} \ [0,T]\times\left[\frac{3}{4}  -\frac{\delta^{\prime\prime}}{2}, \frac{3}{4}  +\frac{\delta^{\prime\prime}}{2}\right], \ \textrm{where} \ C>0.
\end{array}$$
Additionally, since we have the following condition on $G$:
$$G(z)(z-\frac{3}{4} )>0 \textrm{ for } z\in (0,1) \setminus\{\frac{3}{4} \},$$
we will make sure to have 
$$\mathcal{V}(t,z)\left(z-\frac{3}{4} \right) < 0 , \ \forall (t,z) \in (0,T)\times \left((0,1) \setminus \left\{\frac{3}{4} \right\} \right) .$$

\

Let us now compute $\mathcal{V}$ and its first, second and third derivatives:
$$\begin{array}{rcl}
\bar{v}_{tt} &=& \displaystyle{\sum_{i=0}^{3} \ddot{f_i}g_i - 2 \dot{f_i} z \frac{\dot{\lambda}}{\lambda} g_i^\prime - f_i\left(z \frac{\lambda \ddot		{\lambda}-2\dot{\lambda}^2}{\lambda^2} g_i^{\prime}  -z^2\left(\frac{\dot		{\lambda}}{\lambda}\right)^2 g_i^{\prime\prime}\right)} ,\\
\bar{v}_{xx}& = & \lambda^{-2} f_0 G + \displaystyle{\sum_{i=1}^{3} f_i \lambda^{-2} g_i^{\prime\prime}.}
\end{array}$$

$$\begin{array}{rcl}
\lambda^2 \mathcal{V}&=&\displaystyle{-(1-z^2\dot{\lambda}^2)f_0 G + \lambda^2 \ddot{f_0} g_0 - 2 \dot{f_0} z \dot{\lambda}\lambda g_0^\prime-z \left(\lambda \ddot{\lambda}-2\dot{\lambda}^2\right) f_0 g_0^{\prime}  + \sum_{i=1}^{3} \left(\lambda^2 \ddot{f_i}g_i - 2 \dot{f_i} z \dot{\lambda}\lambda g_i^\prime\right. }\\
& &\left. - f_i\left(z \left(\lambda \ddot		{\lambda}-2\dot{\lambda}^2 \right)g_i^{\prime}-z^2\dot		{\lambda}^2 g_i^{\prime\prime}\right)- f_i  g_i^{\prime\prime} \right)\vspace{1em}\\
&=& \displaystyle{-(1-z^2\dot{\lambda}^2)f_0 G + \lambda^2\ddot{f_0} g_0 - 2 \dot{f_0} z \dot{\lambda}\lambda g_0^\prime-z (\lambda \ddot{\lambda}-2\dot{\lambda}^2) f_0 g_0^{\prime}   + \sum_{i=1}^{3} \left( \lambda^2 \ddot{f_i}g_i - 2 \dot{f_i} z \dot{\lambda}\lambda g_i^\prime \right.}\\
& &\displaystyle{ \left.- f_i\left[z (\lambda \ddot		{\lambda}-2\dot{\lambda}^2) g_i^{\prime}  + \left(1-z^2\dot	{\lambda}^2 \right) g_i^{\prime\prime}\right]\right).}
\end{array}$$
Now, for $\e$ small enough (note that this depends on the value of $\delta$), 
\begin{equation}\label{NonZero} 1-(\e z \dot{\lambda_0}(t))^2 > \frac{1}{2}, \quad \forall (t,z) \in [0,T]\times [0,1]\end{equation}
and, using the notation $\lambda$,
\begin{equation}\label{NonZero1}\left\|\frac{1}{1-\left(\frac{3}{4} \dot{\lambda}\right)^2} \right\|_{\ce^2([0,T])}  \leq 10. \end{equation}
Now, if we impose
\begin{equation}\label{1stOrder}\left\{ \begin{array}{l}
\displaystyle{g_i^{(j)}\left(\frac{3}{4} \right)=0 }\ \forall i\in \{1,2,3\} , j \in \{0,1,2\} \ (i,j) \neq (1,2) \\
\displaystyle{g_1^{(2)} \left(\frac{3}{4} \right)=1}
 \end{array}
\right.\end{equation}
and if we define $f_1$ by
\begin{equation}\label{f_1}
f_1:=\frac{1}{1-(\frac{3}{4} \dot{\lambda})^2}\left( \lambda^2 g_0\left(\frac{3}{4} \right)\ddot{f_0} -  2\frac{3}{4} \dot{\lambda}\lambda g_0^\prime\left(\frac{3}{4} \right)\dot{f_0}-\frac{3}{4} (\lambda \ddot{\lambda}-2\dot{\lambda}^2) g_0^{\prime}(\frac{3}{4} ) f_0 \right),
\end{equation}
we get:
$$\lambda^2 \mathcal{V} \left(\ \cdot \ , \frac{3}{4}\right)=0.$$

We now compute the first derivative of $\mathcal{V}$:
$$\begin{array}{rcl}
\lambda^2 \mathcal{V}_z
&=& -(1-z^2\dot{\lambda}^2)f_0 G^\prime+(2z\dot{\lambda}^2-z (\lambda \ddot{\lambda}-2\dot{\lambda}^2))f_0 G -2z \dot{\lambda }\lambda \dot{f_0}G+ \lambda^2\ddot{f_0} g_0^\prime - 2 \dot{f_0} \dot{\lambda}\lambda g_0^\prime\\
& & -(\lambda \ddot{\lambda}-2\dot{\lambda}^2) f_0 g_0^{\prime}  + \sum_{i=1}^{3} \lambda^2 \ddot{f_i}g_i^\prime - 2 \dot{f_i} z \dot{\lambda}\lambda g_i^{\prime\prime}-2 \dot{f_i}  \dot{\lambda}\lambda g_i^{\prime} - f_i\left[ (\lambda \ddot		{\lambda}-2\dot{\lambda}^2) g_i^{\prime}  -2z\dot	{\lambda}^2  g_i^{\prime\prime}\right]  \\
& & - f_i\left[z (\lambda \ddot		{\lambda}-2\dot{\lambda}^2) g_i^{\prime\prime}  + \left(1-z^2\dot	{\lambda}^2 \right) g_i^{(3)}\right]  \vspace{1em}\\
&=& -(1-z^2\dot{\lambda}^2)f_0 G^\prime+(4z\dot{\lambda}^2+z \lambda \ddot{\lambda})f_0 G -2z \dot{\lambda }\lambda \dot{f_0}G+ \lambda^2\ddot{f_0} g_0^\prime - 2 \dot{f_0} \dot{\lambda}\lambda g_0^\prime-(\lambda \ddot{\lambda}-2\dot{\lambda}^2) f_0 g_0^{\prime} \\
& & +\sum_{i=1}^{3} \lambda^2 \ddot{f_i}g_i^\prime - 2 \dot{f_i} \dot{\lambda}\lambda(z  g_i^{\prime\prime}+  g_i^{\prime}) - f_i\left[ (\lambda \ddot		{\lambda}-2\dot{\lambda}^2) g_i^{\prime} +z (\lambda \ddot{\lambda}-4\dot{\lambda}^2) g_i^{\prime\prime}  + \left(1-z^2\dot	{\lambda}^2 \right) g_i^{(3)}\right].
\end{array}$$
Again, we impose
\begin{equation}\label{2ndOrder}\displaystyle{g_i^{(3)}\left(\frac{3}{4} \right)}=\left\{ \begin{array}{l}
0 \ \textrm{if $i\in \{1,3\}$}, \\
1 \ \textrm{if $i=2$},
\end{array}
\right.\end{equation}
and we set
\begin{equation}\label{f_2}
\begin{array}{rcl}
f_2 &:=& \frac{1}{1-(\frac{3}{4} \dot{\lambda})^2}\left[\lambda^2 g_0^\prime\left(\frac{3}{4} \right)\ddot{f_0} -   \dot{\lambda}\lambda g_0^\prime\left(\frac{3}{4} \right)\dot{f_0}- 2  \dot{\lambda}\lambda g_0\left(\frac{3}{4} \right)\dot{f_0}\right.\\
& & \left. -(\lambda \ddot{\lambda}-2\dot{\lambda}^2)  g_0^{\prime} \left(\frac{3}{4} \right) f_0-  2\frac{3}{4} \dot{\lambda}\lambda \dot{f_1}    - \frac{3}{4}  (\lambda \ddot{\lambda}-4\dot{\lambda}^2)f_1 \right]  
\end{array}
\end{equation}
so that 
$$\lambda^2 \mathcal{V}_z\left( \ \cdot \ , \frac{3}{4}  \right) =0.$$ 

Finally,
$$\begin{array}{rcl}
\lambda^2 \mathcal{V}_{zz}
&=&  -(1-z^2\dot{\lambda}^2)f_0 G^{\prime\prime}+(6z\dot{\lambda}^2+z \lambda \ddot{\lambda})f_0 G^\prime -2z \dot{\lambda }\lambda \dot{f_0}G^\prime+ 6\dot{\lambda}^2 f_0 G -4\dot{\lambda }\lambda \dot{f_0}G +\lambda^2\ddot{f_0} G \\
& &+\displaystyle{\sum_{i=1}^{3} \lambda^2 \ddot{f_i}g_i^{\prime\prime} - 2 \dot{f_i} \dot{\lambda}\lambda(2  g_i^{\prime\prime}+ z g_i^{(3)}) - f_i\left[ \left( 2\lambda \ddot{\lambda}-6\dot{\lambda}^2\right) g_i^{\prime\prime}+z (\lambda \ddot{\lambda}-6\dot{\lambda}^2) g_i^{(3)} + \left(1-z^2\dot	{\lambda}^2 \right) g_i^{(4)}\right].} 
\end{array}$$
Again we impose
\begin{equation}\label{3rdOrder}\displaystyle{g_i^{(4)}\left(\frac{3}{4} \right)}=\left\{ \begin{array}{l}
0 \ \textrm{if $i\in \{1,2\}$}, \\
1 \ \textrm{if $i=3$},
\end{array}
\right.\end{equation}
then, by setting:
\begin{equation}\label{f_3}
f_3 = \frac{1}{1-(\frac{3}{4} \dot{\lambda})^2}\left[  -(2\lambda\dot{\lambda} + 2 \dot{\lambda}^2)f_1-4\lambda \dot{\lambda} \dot{f_1} + \lambda^2 \ddot{f_1} -\frac{3}{4} (\lambda \ddot{\lambda}-6\dot{\lambda}^2) f_2  - 2 \frac{3}{4} \dot{\lambda}\lambda \dot{f_2} \right]
\end{equation}
we get:
$$\lambda^2 \mathcal{V}_{zz}\left(\ \cdot \ , \frac{3}{4} \right)=0.$$

Now all that remains is to estimate the third derivative: on $[0,T]\times \left[ \frac{3}{4}-\frac{\delta^{\prime\prime}}{2}, \frac{3}{4} +\frac{\delta^{\prime\prime}}{2}\right]$, by definition of $G$, we have
\begin{equation}\label{VThird}
\lambda^2 \mathcal{V}_{zzz}=-6K_\frac{3}{4}(1-z^2\dot{\lambda}^2)f_0 +\mathcal{R}_0+\mathcal{R},\end{equation}
with:
\begin{equation}\label{R_0}
\mathcal{R}_0   :=   z(8\dot{\lambda}^2 +\ddot{\lambda}\lambda)f_0 G^{\prime\prime} - 2z\dot{\lambda}\lambda \dot{f_0} G^{\prime\prime} + (12\dot{\lambda}^2+ \lambda \ddot{\lambda}) f_0 G^\prime - 6 \dot{\lambda}\lambda \dot{f_0} G^\prime +\lambda^2 \ddot{f_0} G^\prime,
\end{equation}
and
\begin{equation}\label{R}
\mathcal{R}:=\sum_{i=1}^3 \lambda^2 \ddot{f_i} g_i^{(3)} - 2\dot{f_i}\dot{\lambda}\lambda(3g_i^{(3)}+z g_i^{(4)})- f_i\left[(3\ddot{\lambda}\lambda -12 \dot{\lambda}^2)g_i^{(3)} +( z\lambda\ddot{\lambda}-8z\dot{\lambda}^2) g_i^{(4)} + (1-z^2 \dot{\lambda}^2)g_i^{(5)}\right].\end{equation}
Let us note that \eqref{f_0}, combined with the properties of exponential functions, yields
\begin{equation}\label{f0derivatives}
\left(\frac{d}{dt}\right)^n f_0 =  F_n(t) f_0(t), \ \forall n\in \NN,
\end{equation}
where the $F_n$ are rational fractions, the poles of which are $0$ and $T$. 
Now, one can see in \eqref{R_0}, \eqref{f_1}, \eqref{f_2} and \eqref{f_3} that the divergent behaviour of these fractions near $0$ and $T$ is always compensated by the exponential behaviour of $\lambda$ and its derivatives. Furthermore, differentiating the $f_i$ does not change this fact. Hence, keeping \eqref{NonZero1} in mind: 
\begin{equation}\label{f_i_ineq}
\begin{aligned}
\mathcal{R}_0&=\e^2\mathcal{O}(f;t,z), \\
f_1^{(n)}&=\e^2\mathcal{O}(f_0;t), \quad \forall n \in \NN, \\
f_2^{(n)}&=\e^2 \mathcal{O}(f_0;t), \quad \forall n \in \NN,\\
f_3^{(n)}&=\e^4 \mathcal{O}(f_0;t), \quad \forall n \in \NN,
\end{aligned}
\end{equation}
where the notation $\mathcal{O}(f;t)$ (resp. $\mathcal{O}(f;t,z)$) means $f$ times a bounded function of time on $[0,T]$ (resp. time and space).
Hence, near $\frac{3}{4} A$, we have
\begin{equation}\label{g3}\mathcal{R}_0+\mathcal{R}= \e^2 \mathcal{O}(f_0;t,z),\end{equation}
the dominant term being $\ddot{f_1}(1-z^2 \dot{\lambda}^2)g_i^{(5)}$.
Consequently, using \eqref{VThird} and \eqref{g3}, for a small enough $\e$, there exists a constant $C>0$ such that:
\begin{equation}\label{g4}\lambda^2 \mathcal{V}_{zzz} \leq - C f_0 \ \textrm{on} \ \left[\frac{3}{4} -\frac{\delta^{\prime\prime}}{2},\frac{3}{4} +\frac{\delta^{\prime\prime}}{2}\right].\end{equation}
Thus, on $[0,T]\times\left[\frac{3}{4}-\frac{\delta^{\prime\prime}}{2}, \frac{3}{4}+\frac{\delta^{\prime\prime}}{2}\right]$, we can write, thanks to the Taylor-Laplace formula:
$$\lambda^2 \mathcal{V}=\left(z-\frac{3}{4}\right)^3 \varphi(t,z), \ \textrm{with} \ \varphi \in \ce^\infty\left([0,T]\times\left[\frac{3}{4}-\frac{\delta^{\prime\prime}}{2}, \frac{3}{4}+\frac{\delta^{\prime\prime}}{2}\right]\right), \ \varphi <0 \ \textrm{for}\ t\neq 0, T.$$
Additionally, by definition of $f_0$,  $\varphi / \lambda^2$ vanishes exponentially for $t=0,T$, and \eqref{supports} ensures that $\varphi$ vanishes exponentially for $z=1$, so that
$$\left(\frac{\varphi}{\lambda^2}\right)^{\frac{1}{3}} \in \ce^\infty\left([0,T]\times\left[\frac{3}{4}-\frac{\delta^{\prime\prime}}{2}, \frac{3}{4}+\frac{\delta^{\prime\prime}}{2}\right]\right).$$
We now have
\begin{equation}\label{RegMiddle}
\mathcal{V}^{\frac{1}{3}}\in \ce^\infty\left([0,T]\times\left[\frac{3}{4}-\frac{\delta^{\prime\prime}}{2}, \frac{3}{4}+\frac{\delta^{\prime\prime}}{2}\right]\right)\end{equation}

Moreover, on $[0,T] \times \left(\left[0,\frac{3}{4}-\frac{\delta^{\prime\prime}}{2}\right) \cup \left(\frac{3}{4}-\frac{\delta^{\prime\prime}}{2}, 1\right]\right)$, thanks to the constraint on the supports of the $g_i$, we have:
\begin{equation}\label{g0}\lambda^2 \mathcal{V} = \displaystyle{-(1-z^2\dot{\lambda}^2)f_0G + \underbrace{\lambda^2 \ddot{f_0} g_0 - 2 \dot{f_0} z \dot{\lambda}\lambda g_0^\prime-z \left(\lambda \ddot{\lambda}-2\dot{\lambda}^2\right) f_0 g_0^{\prime}      }_{\e^2 \mathcal{O}(f_0;t,x)}.}\end{equation}
As, thanks to Proposition \ref{Stationary1D}, we have 
$$|G|> 2 \ \textrm{on} \ \left[0, \frac{3}{4}-\frac{\delta^{\prime\prime}}{2}\right] \cup \left[\frac{3}{4}+\frac{\delta^{\prime\prime}}{2}, 1-\delta^\prime\right],$$
for small enough $\e$, we have:
$$|\lambda^2 \mathcal{V}| > 0 \ \textrm{on} \ ]0,T[\times\left(\left[0, \frac{3}{4}-\frac{\delta^{\prime\prime}}{2}\right] \cup \left[\frac{3}{4}+\frac{\delta^{\prime\prime}}{2}, 1-\delta^\prime\right]\right).$$
Now, let us recall that, on $[1-\delta^\prime,1)$,
$$\begin{aligned}
g_0(z) &= \displaystyle{e^{-\frac{1}{1-z^2}}},\\
g_0^\prime(z) &= \displaystyle{\frac{-2z}{(1-z^2)^2}e^{-\frac{1}{1-z^2}}},\\
G(z)& = g_0^{\prime\prime}(z) = \displaystyle{\frac{6z^4-2}{(1-z^2)^4}e^{-\frac{1}{1-z^2}}},\\
\end{aligned}$$
So that $g_0/G \ \textrm{and} \ g_0^{\prime}/G$ are bounded near $1$, allowing us to write
\begin{equation}\label{G}\lambda^2 \mathcal{V} = \displaystyle{-f_0 G + \e^2\mathcal{O}(f_0;t)\mathcal{O}_{1^-}(G;z)}.\end{equation}
The notation $\mathcal{O}_{1^{-}}(G;z)$ meaning $G$ times a bounded function of space on $[1-\delta^\prime, 1]$.
So for small enough $\e$, there exists a function $a$ with positive values on $]0,T[$, such that
$$\lambda^2 \mathcal{V}(t,z) \leq -a(t)G(z) < 0, \ \forall (t,z) \in (-0,T)\times [1-\delta^\prime,1).$$
Finally, for all $t\in [0,T]$, $ \mathcal{V} (t, \ \cdot \ )$ vanishes exponentially at $z=1$, and for all $z\in [1-\delta^\prime,1]$, 
$\mathcal{V} ( \ \cdot \ , z)$ vanishes exponentially for $t=0,T$. Hence,
\begin{equation}\label{RegSides}\mathcal{V}^{\frac{1}{3}} \in \ce^\infty\left([0,T]\times\left(\left[0,\frac{3}{4}-\frac{\delta^{\prime\prime}}{2}\right) \cup \left(\frac{3}{4}+\frac{\delta^{\prime\prime}}{2}, 1\right]\right)\right).\end{equation}
This, together with \eqref{RegMiddle}, proves that 
\begin{equation}\label{Regz}\mathcal{V}^{\frac{1}{3}} \in \ce^\infty\left([0,T]\times[0,1]\right).\end{equation}
Now, as $x \mapsto |x|$ is $\ce^\infty$ on $\RR \setminus \{0\}$, by composition we deduce from \eqref{Regz} that 
$$V^{\frac{1}{3}}\in\displaystyle{\ce^\infty\left([0,T] \times \left((0,L)\setminus\{x_0\}\right)\right)}.$$

To deal with the missing point $x_0$, let us recall that for all $t\in]-1,1[$, for all $x\in [0,L]$ such that $|x-x_0|\leq \delta^{\prime\prime} \lambda(t)$ (i.e. $z\leq \delta^{\prime\prime}$),
\begin{equation}\label{SquareNorm}\begin{array}{rcl}
\lambda^2 \mathcal{V}(t,z)&=&-f_0 G + \lambda^2 \ddot{f_0} g_0 - 2 \dot{f_0} z \dot{\lambda}\lambda g_0^\prime-z \left(\lambda \ddot{\lambda}-2\dot{\lambda}^2\right) f_0 g_0^{\prime}  +z^2\dot{\lambda}^2 f_0 G \\
&=& 2 f_0 + \ddot{f_0} (\lambda^2 - |x-x_0|^2) +4  \dot{f_0}  \frac{\dot{\lambda}}{\lambda} |x-x_0|^2+2 \left(\frac{\ddot{\lambda}}{\lambda }-2\left(\frac{\dot{\lambda}}{\lambda}\right)^2\right) f_0 |x-x_0|^2  -2 \left(\frac{\dot{\lambda}}{\lambda}\right)^2 f_0 |x-x_0|^2\\
&=& 2  f_0+\lambda^2 \ddot{f_0} + \psi(t)|x-x_0|^2,
\end{array}\end{equation}
where $\psi\in\ce^\infty([0,T])$, and $\psi$ vanishes exponentially for $t=0,T$, along with all its derivatives.

\noindent We now see that the terms in $|x-x_0|$ of $V$ are actually in $|x-x_0|^2$, which compensates the singularity at $0$ of the map $x \mapsto |x|$. Thus, from the smoothness of $\mathcal{V}^{\frac{1}{3}}$ we get, by composition, $V^{\frac{1}{3}} \in \ce^\infty\left([0,T] \times [0,L] \right)$. 
Thus we have proved that, by chosing $g_i$ that verify \eqref{supports}, \eqref{1stOrder}, \eqref{2ndOrder} and \eqref{3rdOrder}, we get
$$V^{\frac{1}{3}} \in \ce^\infty\left([0,T] \times [0,L] \right).$$

Finally, we set 
$$\begin{aligned}
\bar{v}(x,t) &:=\sum_{i=0}^3 f_i(t) g_i\left(\frac{|x-x_0|}{\lambda(t)}\right), \\
\bar{u} &:=(\wave \bar v)^{\frac{1}{3}}, \\
\bar{h} &:=\wave \bar{u}.
\end{aligned}$$
where $\lambda$ is defined by \eqref{lambda}, the $g_i$ are some functions satisfying \eqref{supports}, \eqref{1stOrder}, \eqref{2ndOrder}, and \eqref{3rdOrder}, and the $f_i$ are defined by \eqref{f_0}, \eqref{f_1}, \eqref{f_2}, and \eqref{f_3}. 

Let us check that we have indeed built a return trajectory: for $i\in\{0, \cdots , 3\}$, the $f_i$ vanish at $-1$ and $1$, along with all their derivatives. Hence, 
$$\bar{u}(-1, \ \cdot \ )=\bar{v}(-1, \ \cdot \ )=\bar{u}_t(-1, \ \cdot \ )=\bar{v}_t(-1, \ \cdot \ )=0,$$
$$\bar{u}(1, \ \cdot \ )=\bar{v}(1, \ \cdot \ )=\bar{u}_t(1, \ \cdot \ )=\bar{v}_t(1, \ \cdot \ )=0.$$\qed

\begin{Rq}\label{VanishingPoints}
Most of the work in the construction above comes from the vanishing points $\displaystyle{\left(t, (3/4) \lambda(t)\right)}$ ``in the middle of the domain''. So one could wonder, would it not be simpler to try and build a function that only vanishes, along with all its derivatives, at the points $\left(t, \lambda(t) \right)$? 

Let us remind that our strategy to build the return trajectory is to start from a solution to the stationary problem, and then make it evolve through time so as to stay ``not too far away from it''. But the reason we have vanishing points ``in the middle of the domain'' has to do with that same stationary problem. More precisely, the stationary problem consists in finding functions that vanish, along with their derivatives, on the boundary of the domain. In our case this condition corresponds to
\begin{equation}
\label{SmoothVanish}
g(z) = e^{-\frac{1}{1-z^2}} \ \textrm{on} \ [1-\delta^\prime, 1].\end{equation}

We further require that the Laplacians of these functions be third powers of  $\ce^\infty$ functions. In our case this condition becomes
$$\begin{array}{c}
\displaystyle{G(z)\left(z-\frac{3}{4} \right)>0}, \\
\displaystyle{G(z)=\left(z-\frac{3}{4}\right)^3} \ \textrm{on} \  \left[\frac{3}{4} -\frac{\delta^{\prime\prime}}{2}, \frac{3}{4} + \frac{\delta^{\prime\prime}}{2}\right].
\end{array}$$

Now, we could instead demand that $G$ be non-negative (or non-positive). But then, by convexity arguments (or Hopf's maximum principle), we would get 
$$g^\prime (1)<0,$$

Which contradicts condition \eqref{SmoothVanish}. But that condition is very helpful in
proving the smoothness of $\mathcal{V}^{\frac{1}{3}}$ near the boundary. Giving it
up would mean setting more conditions on the $g_i$ functions near the boundary, so
we would have to give up condition \eqref{supports}, and then set additional
conditions on the $g_i$ to make sure $V$ is well defined (as $\lambda(0)=\lambda(T)=0$),
preserve the sign of $\mathcal{V}$ or more generally its smoothness, in particular
near the boundary...Which would probably be more trouble than what we had to do at
the vanishing points $\displaystyle{\left(t, (3/4) \lambda(t)\right)}$.
\end{Rq}

\subsection{Covering sets and return trajectories}\label{finalreturn}

\

As mentioned at the beginning of this section, we want to work on a smooth subset of $[0,T]\times [a,b]$ where $u \neq 0$. However, to do so we need more than the elementary trajectory described above: rather, we use the elementary trajectory as a building block for our final return trajectory. Indeed, let $0<\delta<\min\left((b-a)/4,T/2\right) $ such that \eqref{delta} is satisfied. The preliminary construction gives us a real number $\e >0$ (after the right rescaling of the space variable) and, for any  $x_0 \in [a+\delta+\e,b-\delta-\e]$, a trajectory $(\bar{u}, \bar{v}, \bar{h})$ such that $\bar{u} \neq 0$ on $\Lambda_{\e, x_0}:=\{(t,x) \ | \ |x-x_0| < (3/4) \e \lambda_0(t)\}$, which contains any rectangle of the form $[\delta, T-\delta] \times [x_0 - \xi, x_0+\xi]$ with $\xi< (3/4) \e$. Moreover, each of these rectangles can be fit into the interior of a smooth closed subset of $\Lambda_{\e, x_0}$. 
\begin{figure}[h!]
\centerline{\includegraphics[scale=0.5]{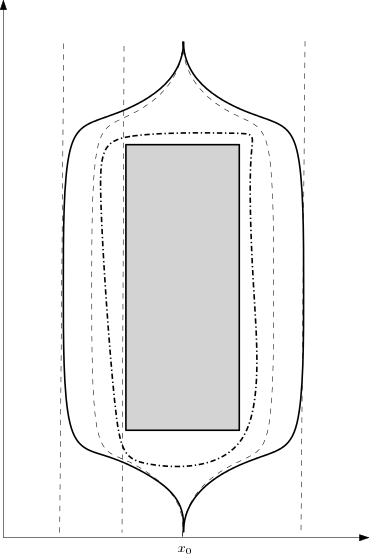}}
		\caption{\label{fig:Tuile} \textit{The support of the preliminary construction with a rectangle fit inside the line of vanishing points of $\bar{u}$.}}
\end{figure}

 Now there are cases (if $[a,b]$ is too long and $\e$ - and consequently, $\xi$ - too small), where none of the rectangles $[\delta, T-\delta] \times [x_0 - \xi, x_0+\xi]$ satisfies the Geometric Control Condition (GCC). Thus we cannot apply Proposition \ref{InternalTwoControls1} with controls supported in some $[\delta, T-\delta] \times [x_0 - \xi, x_0+\xi]$, as time condition \eqref{TimeCond1} does not hold in these cases. So we need to build a return trajectory $(\bar{u}, \bar{v} , \bar{h})$ such that $\bar{u} \neq 0$ on a smooth closed set $\mathcal{Q}$ containing a set $\mathcal{Q}_\delta$ that satisfies the GCC.
 
 Now there is a simple type of set that would fit our needs for $\mathcal{Q}_\delta$: in Section 2 we worked in $[\delta, T-\delta]\times [a+\delta, b+\delta]$, but we do not need the whole rectangle in general for the GCC to be satisfied. We can in fact work with a number of much smaller rectangles, as long as they are close enough to each other:
 \begin{Def}\label{CoveringSets}
 Let $0<\delta<\min\left((b-a)/4,T/2\right) $, such that 
 \eqref{delta} is satisfied. 
 A \textbf{$\delta$-covering set} of $[0,T] \times [a,b]$ for system \eqref{SystemGeneral} is a union of rectangles of the form $\left\{[\delta, T-\delta] \times [a_i,b_i], 1\leq i\leq N \right\}$ for some $N \geq 1$, such that
 \begin{equation}
 \begin{gathered}
 \begin{aligned}
 a_1 & =a+\delta, \\
 b_N & =b-\delta,
 \end{aligned}\\
 0<(a_{i+1}-b_i) \max \left(\frac1\nu_1, \frac1\nu_2\right) <T-2\delta, \quad 1 \leq i \leq N-1.
 \end{gathered}
 \end{equation}
 \end{Def}
 	
 		\begin{figure}[h]
 		\centerline{\includegraphics[scale=0.75]{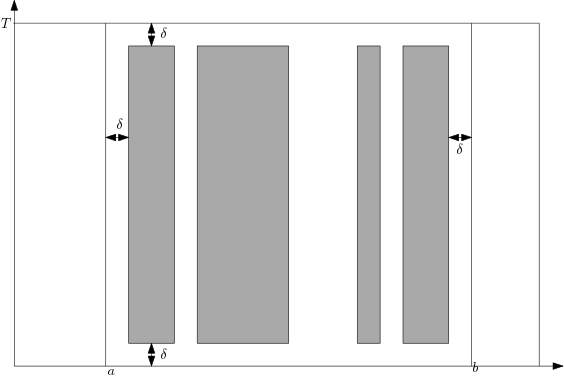}}
 		\caption{\label{fig:DeltaCoveringSet} An example of a $\delta$-covering set.}
 		\end{figure}

Now the idea is to add the elementary trajectories obtained by the preliminary construction on disjoint supports centered in $x_i \in [a+\delta+\e, b-\delta-\e]$, that are close enough, and with a small enough $\e$ so that the rectangles $[\delta, T-\delta] \times [x_0 - \e/2, x_0+\e/2]$ form a $\delta$-covering set. Take $\e_0 \leq (b-a-2\delta)/2$ small enough for the preliminary construction to work, and such that $\e_0\max\left(1/\nu_1, 1/\nu_2\right) < T-2\delta$. We then define the following sequence: take $N \in \NN$ large enough so that
$$ \e:= \frac{b-a-2\delta}{2N-1} \leq \e_0$$
and define, for $1\leq i \leq N$,
$$x_i:=a+\delta + \left(2i-\frac32\right) \e,$$
and $(\bar{u}_i, \bar{v}_i, \bar{h}_i)$ the trajectory obtained by the preliminary construction corresponding to the chosen $\e$, centered in $x_i$. Let $\mathcal{Q}_i$ be a smooth closed subset of $\Lambda_{\e, x_i}$ containing $[\delta, T-\delta] \times [x_i - \e/2, x_i+\e/2]$ in its interior.
Then, 
$$\mathcal{Q}_\delta:= \bigcup_{i=1}^N [\delta, T-\delta] \times [x_i - \e/2, x_i+\e/2]$$
is a $\delta$-covering set, 
$$\mathcal{Q}:=\left(\bigcup_i \mathcal{Q}_i \right) $$
is a smooth closed set such that $\mathcal{Q}_\delta \subset \overset{\circ}{\mathcal{Q}}$, and we can define	
	\begin{equation}\label{ReturnTrajectoryFinal}(\bar{u}, \bar{v}, \bar{h}):=\sum_{i=1}^N (\bar{u}_i, \bar{v}_i, \bar{h}_i),\end{equation}
		which is supported in $[0,T]\times[a,b]$, and 
		satisfies \eqref{uSuitable}.

\begin{figure}[!h]
\centerline{\includegraphics[scale=1]{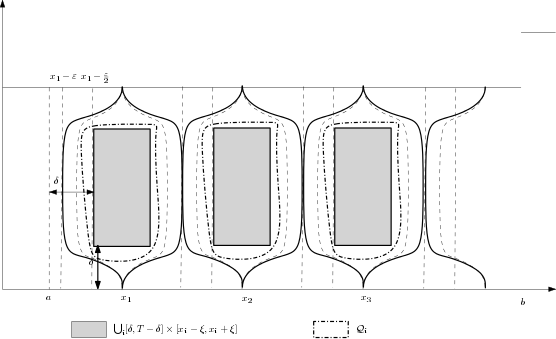}}
		\caption{\label{fig:FinalSupport} \textit{Putting elementary trajectories side by side. The rectangles form a covering set.}}
		\end{figure}

  \subsection{Local controllability with two controls and Gromov inversion}\label{2controls}
  
  \
  
 We now have our return trajectory $(\bar{u}, \bar{v}, \bar{h})$. Now let $R>0$, and notice that for all $\kappa>0$, $(\kappa\bar{u}, \kappa^3\bar{v}, \kappa\bar{h})$ is also a return trajectory, with the same support. Thus, we can now suppose without loss of generality, that 
 \begin{equation}\label{SmallReturn}\|(\bar{u}, \bar{v}, \bar{h})\|_{\left(\ce^{11}\right)^3}\leq \frac{R}2.\end{equation}
 Let $u,v \in \ce^k([0,T]\times [0,L])$, $h_1, h_2 \in \ce^{k-2}([0,T]\times [0,L])$. Let us consider the trajectory $(\bar{u}+u, \bar{v}+v)$, controlled by $(\bar{h}+h_1, h_2)$, we get the following control system for $u$ and $v$:
  \begin{equation}\label{TwoControlsLocal}
  \left\{\begin{aligned}
  \wave_{\nu_1} u & = h_1, \\
  \wave_{\nu_2} v & = u^3+ 3\bar{u} u^2 + 3 \bar{u}^2 u + h_2, \\
  u( \ \cdot \ , 0) & = 0,  \\
  u( \ \cdot \ , L) & = 0, \\
  v( \ \cdot \ , 0) & = 0, \\
  v( \ \cdot \ , L) & = 0. 
  \end{aligned} \right.\end{equation}
  This is a coupled semilinear system with a source term, and falls in the category of systems \eqref{SystemGeneral2controls}. The aim of this section is to prove the following proposition:
  \begin{Prop} \label{InternalTwoControls}
  Let $k \geq 2$, $0 \leq a < b \leq L$, $T>0$ such that
  $$T > 2(L-b)\max\left(\frac{1}{\nu_1}, \frac{1}{\nu_2}\right), \ T>2a\max\left(\frac{1}{\nu_1}, \frac{1}{\nu_2}\right).$$
  For every $0<\delta <\min\left(T/2,(b-a)/2\right)$ satisfying \eqref{delta},
  for every $\delta$-covering set $\mathcal{Q}_\delta$ of $[0,T] \times [a,b]$, there exists $\eta >0$ and constants $C_1, C_2>0$ depending on $T, \delta, k$ such that, for initial and final values
  $$((u_0,u_1),  (v_0, v_1), (u_0^{f}, u_1^{f}),(v_0^f, v_1^f) ) \in \left(B_{\ce^{k}([0,T]\times [0,L])}(0,\eta) \times B_{\ce^{k-1}([0,T]\times [0,L])}(0,\eta)\right)^4$$
  satisfying \eqref{CoupleCompatibility} at the order $k$,
  there exist controls $h_1, h_2 \in \ce^{k-1}([0,T]\times [0,L])$ satisfying 
  \begin{align}
  \label{ControlSupports2}
  &\textup{supp} \ h_i \subset \mathcal{Q}_\delta,
  &i=1,2,\\
  \label{hSmallness2}
  &\|h_i\|_{\ce^{k-2}} \leq C_1 \|((u_0,u_1),  (v_0, v_1), (u_0^{f}, u_1^{f}),(v_0^f, v_1^f) )\|_{\left(\ce^k \times \ce^{k-1}\right)^4},
  &i=1,2,
  \end{align}
  such that the corresponding solution of \eqref{SystemGeneral2controls} with initial values $((u_0, u_1), (v_0, v_1))$ satisfies
 \begin{equation}\left\{\begin{aligned} u(T,\ \cdot \ )&=u_0^{f},&
 u_t(T,\ \cdot \ )&=u_1^{f},\\
 v(T, \ \cdot \ )&=v_0^f,&
 v_t(T, \ \cdot \ )&=v_1^f.\end{aligned}\right.\end{equation}
  \begin{equation}\label{ySmallness2}
  \|(u,v)\|_{(\ce^k)^2} \leq C_2 \|((u_0,u_1),  (v_0, v_1), (u_0^{f}, u_1^{f}),(v_0^f, v_1^f) )\|_{\left(\ce^k \times \ce^{k-1}\right)^4} .
  \end{equation}
  \end{Prop}
  
  \begin{Rq}
  It is clear, by Definition \ref{CoveringSets}, that for any $0<\delta<\min\left((b-a)/4,T/2\right) $ such that 
  \eqref{delta} is satisfied, $[\delta, T-\delta]\times [a+\delta, b-\delta]$ is a $\delta$-covering set of $[0,T]\times [a,b]$. Thus Proposition \ref{InternalTwoControls} implies Proposition \ref{InternalTwoControls1}.
  \end{Rq}

  To prove this proposition, we use the following propositions, which are particular cases of more general quasilinear results proved in \cite{LiRao} (see also \cite[chapter 5, sections 5.3 and 5.4]{Li}):
  \begin{Prop}[two--sided control]\label{LiBoundary2}
  Let $k\geq 2$, $L>0$, $T>0$, $F\in \ce^\infty(\RR^2, \RR^2)$, $\nu_1, \nu_2>0$. If
  $$T>L\max\left(\frac{1}{\nu_1}, \frac{1}{\nu_2}\right),$$
  then there exists $\eta >0$ and a constant $C>0$ depending on $T, k$, such that for any initial and final values
  $$(U_0,U_1, U_0^f, U_1^f) \in B_{\left(\ce^k ([0,L])^2 \times \ce^{k-1}([0,L])^2\right)^2}(0, \eta)$$
  there exist controls $H_1$ and $H_2\in \ce^k([0,T], \RR^2)$ satisfying compatibility conditions
  
  \begin{equation}\label{2BoundaryCompatibility}
\left\{ \begin{aligned}
P_{n,i}^{f_i}\left(J_x^n(U_0)(0), J_x^{n-1}(U_1)(0), (0,\cdots,0)\right)&=\partial_t^n H_{1i}(0), \\
P_{n,i}^{f_i}\left(J_x^n(U_0)(L), J_x^{n-1}(U_1)(L), (0,\cdots,0)\right)&=\partial_t^n H_{2i}(0), \\
P_{n,i}^{f_i}\left(J_x^n(U_0^f)(0), J_x^{n-1}(U_1^f)(0),  (0,\cdots,0)\right)&=\partial_t^n H_{1i}(T),\\
P_{n,i}^{f_i}\left(J_x^n(U_0^f)(L), J_x^{n-1}(U_1^f)(L), (0,\cdots,0)\right)&=\partial_t^n H_{2i}(T),
\end{aligned}\right.
\quad \forall n\leq k, i=1,2.
  \end{equation}
  such that the solution to the vector system
  \begin{equation}\label{TwoControls2sided}\left\{
  \begin{aligned}
  \partial_{tt} U -\begin{pmatrix} \nu_1^2 & 0 \\ 0 & \nu_2^2 \end{pmatrix}\partial_{xx} U & = F(U) , \ x \in (0,L)\\
  U(t,0)&=H_1, \\
  U(t,L)&=H_2, \\
  U(0)&=U_0,\\
  U_t(0)&=U_1,
  \end{aligned}
  \right.
  \end{equation}
   satisfies
  $$\left\{\begin{aligned}
   U(T)&=U_0^f,\\
   U_t(T)&= U_1^f, 
    \end{aligned}\right.$$
  \begin{equation}
  \label{Smallness2side}
  \|U\|_{\ce^k}  \leq C\|(U_0,U_1), (U_0^f, U_1^f)\|_{\left(\ce^k \times \ce^{k-1}\right)^2}.\end{equation}

  \end{Prop}
  \begin{Prop}[one-sided control]\label{LiBoundary1}
  Let $k\geq 2$, $L>0$, $T>0$, $F\in \ce^\infty(\RR^2, \RR^2)$, $\nu_1, \nu_2>0 $. If
  $$T>2L\max\left(\frac{1}{\nu_1}, \frac{1}{\nu_2}\right),$$
  then there exists $\eta >0$ and a constant $C>0$ depending on $T, k$, such that for any initial and final values
  $$(U_0,U_1, U_0^f, U_1^f) \in B_{\left(\ce^k ([0,L])^2 \times \ce^{k-1}([0,L])^2\right)^2}(0, \eta)$$
   there exists a control $H\in \ce^k([0,T], \RR^2)$ satisfying compatibility conditions
   \begin{equation}\label{1BoundaryCompatibility}
   \left\{ \begin{aligned}
P_{n,i}^{f_i}\left(J_x^n(U_0)(0), J_x^{n-1}(U_1)(0), (0,\cdots,0)\right)&=\partial_t^n H_{i}(0) \quad (resp. \ 0), \\
P_{n,i}^{f_i}\left(J_x^n(U_0)(L), J_x^{n-1}(U_1)(L), (0,\cdots,0)\right)&=0 \quad (resp. \partial_t^n H_{i}(0)), \\
P_{n,i}^{f_i}\left(J_x^n(U_0^f)(0), J_x^{n-1}(U_1^f)(0),  (0,\cdots,0)\right)&=\partial_t^n H_{i}(T) \quad (resp. \ 0),\\
P_{n,i}^{f_i}\left(J_x^n(U_0^f)(L), J_x^{n-1}(U_1^f)(L), (0,\cdots,0)\right)&=0 \quad (resp. \partial_t^n H_{i}(T)),
\end{aligned}\right.
\quad \forall n\leq k, i=1,2.
   \end{equation}
   such that the solution to the vector system
  \begin{equation}\label{TwoControls1sided}\left\{
  \begin{aligned}
  \partial_{tt} U -\begin{pmatrix} \nu_1^2 & 0 \\ 0 & \nu_2^2 \end{pmatrix}\partial_{xx} U & = F(U) , \ x \in (0,L),\\
  U(t,L)&=0 \ (resp. \ U(t,L)=H(t)), \\
  U(t,0)&=H, \ (resp. \ U(t,0)=0), \\
  U(0)&=U_0,\\
  U_t(0)&=U_1,
  \end{aligned}
  \right.
  \end{equation}
   satisfies
  $$ \left\{\begin{aligned}
   U(T)&=U_0^f,\\
   U_t(T) &= U_1^f, 
   \end{aligned}\right.$$
   \begin{equation}
   \label{Smallness1side}
  \|U\|_{\ce^k} \leq C\|(U_0,U_1), (U_0^f, U_1^f)\|_{\left(\ce^k \times \ce^{k-1}\right)^2}.   \end{equation}
  
  \end{Prop}
  \begin{proof}[Proof of Proposition \ref{InternalTwoControls}]
  Let us note 
  $$\mathcal{Q}_\delta=\bigcup_{1\leq i \leq N} [\delta, T-\delta] \times [a_i, b_i],$$
  for some $N\geq 1$. for every $1\leq i \leq N-1$, let $0<\delta_i < \min ((b_{i+1}-a_{i+1})/2,(b_i-a_i)/2)$ such that
  \begin{equation}\label{deltai}T-2\delta_i > (a_{i+1}-b_i+4\delta_i)\max\left(\frac{1}{\nu_1}, \frac{1}{\nu_2}\right).\end{equation}
  Thanks to Propositions \ref{LiBoundary2} and \ref{LiBoundary1}, Definition \ref{CoveringSets} and conditions  \eqref{delta} and \eqref{deltai}, there exists $\eta>0$ such that for initial and final values 
  $$((u_0,u_1),  (v_0, v_1), (u_0^{f}, u_1^{f}),(v_0^f, v_1^f) )\in \left(B_{\ce^k([0,L])}(0, \eta) \times B_{\ce^{k-1}([0,L])}(0, \eta)\right)^4$$ 
  satisfying \eqref{CoupleCompatibility}, 
  \begin{itemize}
  \item There exist boundary controls $u_1^{(i)}, u_2^{(i)} \in \ce^k([0,T-2\delta])$ at $b_i-\delta_i$ and $a_{i+1}+\delta_i$ that steer $(u,v)$ on $[b_i-\delta_i,a_{i+1}+\delta_i]$ from $(y_0,y_1)_{|[b_i-\delta_i,a_{i+1}+\delta_i]}$ to $(z_0,z_1)_{|[b_i-\delta_i,a_{i+1}+\delta_i]}$. \item There exist two boundary controls $u_1, u_2 \in \ce^k([0,T-2\delta])$ at $a+2\delta$ and $b-2\delta$ that steer $(u,v)$ on $[0,a+2\delta]$ from $(y_0,y_1)_{|[0,a+2\delta]}$ to $(z_0,z_1)_{|[0,a+2\delta]}$, and from $(y_0,y_1)_{|[b-2\delta, L]}$ to $(z_0,z_1)_{|[b-2\delta, L]}$ while satisfying the boundary conditions of the system at $0$ and $L$.
  \end{itemize}
  \begin{figure}[h!]
 \centerline{\includegraphics[scale=0.75]{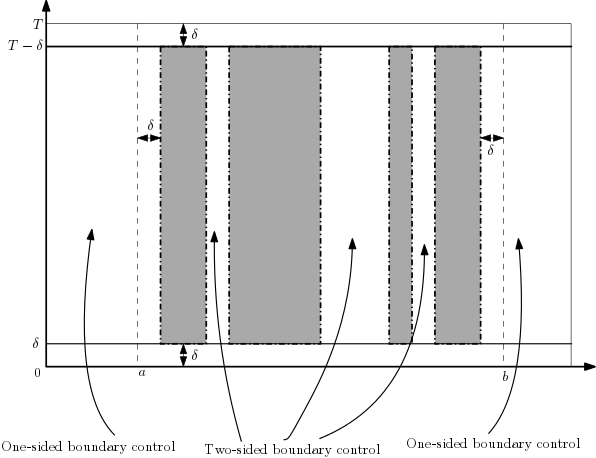}}
 		\caption{\label{fig:Stationary} \textit{Using boundary control results outside of the covering set.}}
 \end{figure}
  We note $\delta_0=\delta_N:=\delta$, and $(u^\ast, v^\ast)$ the corresponding trajectory on $[0,a+2\delta]\cup [b-2\delta, L] \cup \bigcup_{1\leq i \leq N-1} [b_i-\delta_i,a_{i+1}+\delta_i]$. Then, \eqref{Smallness1side} and \eqref{Smallness2side} imply
  \begin{equation}\label{estimateStar}
  \|(u^\ast, v^\ast)\|_{\ce^k} \leq C\|(u_0,u_1), (v_0,v_1),(u_0^f, u_1^f), (v_0^f, v_1^f)\|_{\left(\ce^k \times \ce^{k-1}\right)^4} \end{equation}
  for some constant $C>0$.
  
  On the other hand, for $\eta>0$ small enough, for initial and final conditions
  $$((u_0,u_1),  (v_0, v_1), (u_0^{f}, u_1^{f}),(v_0^f, v_1^f) ) \in \left(B_{\ce^{k}([0,T]\times [0,L])}(0,\eta) \times B_{\ce^{k-1}([0,T]\times [0,L])}(0,\eta)\right)^4,$$
   for $1\leq i \leq N$ the forward evolving solutions $\left(u_f^{(i)}, v_f^{(i)}\right)$ of the vector equations  
  $$\left\{\begin{aligned}
  \wave_{\nu_1}u&= f_1(\bar{u}+u, \bar{v}+v)-f_1(\bar{u},\bar{v}), \\
  \wave_{\nu_2} v&= f_2(\bar{u}+u, \bar{v}+v)-f_2(\bar{u},\bar{v}), \\
  (u,v)(t,a_i)&=(u,v)(t,b_i)= (0,0), \\
  (u,v)(0,\ \cdot \ )&= (u_{0|[a_i,b_i]}, v_{0|[a_i,b_i]}), \\
  (u,v)_t(0, \ \cdot \ )&= (u_{1|[a_i,b_i]},v_{1|[a_i,b_i]}).
  \end{aligned}\right.,$$
  are defined on $[0, T-2\delta] \times [a_i, b_i]$.
   Let us also note $\left(u_b^{(i)}, v_b^{(i)}\right)$ the backward evolving solutions of the vector equations on $[0, T-2\delta] \times [a_i, b_i]$
   $$\left\{\begin{aligned}
  \wave_{\nu_1} u&= f_1(\bar{u}+u, \bar{v}+v)-f_1(\bar{u},\bar{v}), \\
  \wave_{\nu_2} v&= f_2(\bar{u}+u, \bar{v}+v)-f_2(\bar{u},\bar{v}), \\
  (u,v)(t,a_i)&=(u,v)(t,b_i)= (0,0), \\
   (u,v)(T-2\delta,\ \cdot)&= (u_{0|[a_i,b_i]}^f, v_{0|[a_i,b_i]}^f), \\
   (u,v)_t(T-2\delta, \ \cdot \ )&= (u_{1|[a_i,b_i]}^f, v_{1|[a_i,b_i]}^f).
   \end{aligned}\right.$$
   Then we define $(\tilde{u}, \tilde{v})$ by
   $$(\tilde{u}, \tilde{v})= \left(u_f^{(i)}, v_f^{(i)}\right) \phi + \left(u_b^{(i)}, v_b^{(i)}\right)(1-\phi) , \textrm{on} \ [a_i, b_i], \forall i \leq N, $$
   where $\phi$ is a time cut-off function such that
   $$\phi(0)=1, \quad \phi(T-2\delta)=0.$$
  Note that, by well-posedness of the Cauchy problems, there exists $C^\prime>0$ such that the norm of $(\tilde{u}, \tilde{v})$ satisfies
  \begin{equation}\label{estimateTilde}
  \|(\tilde{u}, \tilde{v})\|_{(\ce^k)^2} \leq C^\prime \|(u_0,u_1), (v_0,v_1),(u_0^f, u_1^f), (v_0^f, v_1^f)\|_{\left(\ce^k \times \ce^{k-1}\right)^4} . \end{equation}
  Finally, let us define $(u^{\ast\ast}, v^{\ast\ast})$ by smoothly extending $(u^\ast, v^\ast)$ on $\bigcup_{1\leq i \leq N}[a_i+\delta_{i-1}, b_i-\delta_i]$ with
  \begin{equation}\label{ExtensionEstimate}\|(u^{\ast\ast}, v^{\ast\ast})\|_{(\ce^k)^2} \leq C^{\prime\prime} \|(u^\ast, v^\ast)\|_{\ce^k} ,\end{equation}
  where $C^{\prime\prime}$ is a constant depending on the $a_i, b_i$.
   Then, we define $(u,v)$ by
   $$(u,v)=\xi(u^{\ast\ast}, v^{\ast\ast}) + (1-\xi) (\tilde{u}, \tilde{v}),$$
   where $\xi$ is a space cut-off function satisfying
   $$\xi=1 \ \textrm{on} \ [0, a+\delta] \cup [b-\delta, L] \cup \bigcup_{1\leq i \leq N-1} [b_i,a_{i+1}] ,$$
   $$\xi=0 \ \textrm{on} \ \bigcup_{1\leq i \leq N} [a_i+\delta_{i-1}, b_i-\delta_i].$$
   
   Then, by construction, we have
   $$\left\{\begin{aligned}
   u(0,\ \cdot \ )&= u_0,&
   v(0,\ \cdot \ )&= v_0, \\
   u_t(0, \ \cdot \ )&= u_1,& 
    v_t(0, \ \cdot \ )&= v_1, \\
   u(T-2\delta, \ \cdot \ )&= u_0^f,& 
   v(T-2\delta, \ \cdot \ )&= v_0^f, \\
   u_t(T-2\delta,\ \cdot \ )&= u_1^f,&
   v_t(T-2\delta,\ \cdot \ )&= v_1^f,
   \end{aligned}\right.$$
   and
  
  $$\begin{aligned}
   \textrm{supp} \  (\wave_{\nu_1} u-f_1(\bar{u}+u, \bar{v}+v)- f_1(\bar{u}, \bar{v})) &\subset \mathcal{Q}_\delta, \\
  \textrm{supp} \ (\wave_{\nu_2} v-f_2(\bar{u}+u, \bar{v}+v)- f_2(\bar{u}, \bar{v})) &\subset \mathcal{Q}_\delta,\end{aligned}$$
   Finally, \eqref{estimateStar}, \eqref{estimateTilde} and \eqref{ExtensionEstimate} imply that there exists a constant $C_2 >0$ such that \eqref{ySmallness2} holds, and, by continuity of the $f_i$, noting
$$h_i:= \wave_{\nu_i} u-f_i(\bar{u}+u, \bar{v}+v)- f_i(\bar{u}, \bar{v}), \quad i=1,2,$$
there exists a constant $C_1>0$ such that \eqref{hSmallness2} holds.
 \end{proof}
	Now, we define
		$$\mathcal{A}=\left\{(u,v,h) \in \left(\ce^{2}(\mathcal{Q})\right)^3 \ | \ \forall (t,x) \in \mathcal{Q}, \ u(t,x) \neq 0\right\},$$
		which is clearly nonempty, and
		$$\forall (u,v,h) \in \ce^2(\mathcal{Q})^3, \De(u,v,h)=\left(\wave_{\nu_1} u - h, \wave_{\nu_2} v - u^3\right).$$
		Then, we have the following proposition, similar to Proposition \ref{InfinitesimalInversion}:
		\begin{Prop}\label{InfinitesimalInversion2}
		$\mathscr{D}$ admits an infinitesimal inversion of order $2$ over $\mathcal{A}$.
		\end{Prop}
		Moreover, thanks to \eqref{ReturnTrajectoryFinal} and \eqref{uSuitable},
		\begin{Prop}
			$$(\bar{u}, \bar{v}, \bar{h})_{|\mathcal{Q}} \in \mathcal{A}.$$
		\end{Prop}
		
		Now, we can use Theorem \ref{GromovInversion}: there exists $\eta>0$ such that for initial and final conditions 
        $$((u_0,u_1),  (v_0, v_1), (u_0^{f}, u_1^{f}),(v_0^f, v_1^f) ) \in \left(B_{\ce^{11}([0,L])}(0,\eta) \times B_{\ce^{10}( [0,L])}(0,\eta)\right)^4$$
the corresponding trajectories of system \eqref{TwoControlsLocal} with two controls $u^\ast, v^\ast, h_1, h_2$ are small enough in $(\ce^{11})^2 \times (\ce^9)^2$ norm so that $\mathscr{D}$ can be inverted locally around $(\bar{u}+u^\ast, \bar{v}+v^\ast, \bar{h})$, and so that, by the continuity property, $(u,v,h):=\mathscr{D}^{-1}_{(\bar{u}+u^\ast, \bar{v}+v^\ast, \bar{h})} (\theta_1, \theta_2)$ satisfies
		\begin{equation}
		\|(u-\bar{u},v-\bar{v},h-\bar{h})\|_{(\ce^6)^3} \leq \frac{R}2.
		\end{equation}
		Together with \eqref{SmallReturn}, this yields
        \begin{equation}\label{LocalEstimate}
\|(u,v,h)\|_{(\ce^6)^3} \leq R.
\end{equation}
This proves the following local controllability result:
\begin{Th}\label{Mainresult}
Let $R>0$, and $0\leq a < b \leq L$, $T>0$ such that
\begin{equation}\label{TimeCond1} T > 2(L-b)\max\left(\frac{1}{\nu_1}, \frac{1}{\nu_2}\right), \ T>2a\max\left(\frac{1}{\nu_1}, \frac{1}{\nu_2}\right).\end{equation}
There exists $\eta >0$ such that for given initial and final conditions 
$$((u_0,u_1),  (v_0, v_1), (u_0^{f}, u_1^{f}),(v_0^f, v_1^f) ) \in \left(B_{\ce^{11}([0,L])}(0,\eta) \times B_{\ce^{10}( [0,L])}(0,\eta)\right)^4$$
satisfying \eqref{CoupleCompatibility}, there exists $h \in \ce^6 ([0,T]\times[0,L])$ satisfying
$$\textup{supp} \ h \subset [0, T] \times [a, b].$$
such that the corresponding solution $(u,v)\in \ce^6 ([0,T]\times[0,L])$ of \eqref{System} with initial values $((u_0,u_1), (v_0, v_1))$ satisfies
$$\left\{\begin{aligned}u(T,\ \cdot \ )=u_0^{f}, \ &
u_t(T,\ \cdot \ )=u_1^{f},\\
v(T, \ \cdot \ )=v_0^f, \ &
v_t(T, \ \cdot \ )=v_1^f\end{aligned}\right.$$
and \eqref{LocalEstimate} holds.


\end{Th}
 
  Now let $(u_0,u_1, v_0, v_1, u_0^f, u_1^f, v_0^f, v_1^f) \in \left(\ce^{11}([0,L]) \times \ce^{10}([0,L])\right)^4$ such that \eqref{CoupleCompatibility} is satisfied. Let us note
$$M:=\|(u_0,u_1,u_0^f,u_1^f)\|_{\left(\ce^{11}\times\ce^{10}\right)^2}+\|(v_0,v_1,v_0^f,v_1^f)\|^{\frac13}_{\left( \ce^{11}\times\ce^{10}\right)^2},$$ 
and $\alpha:=\frac{\eta}{2M}$. Then, 
$$\begin{aligned}
\|\alpha u_0\|_{\ce^{11}} &\leq \eta, &\|\alpha u_1\|_{\ce^{10}} &\leq \eta, &
\|\alpha u_0^f\|_{\ce^{11}} &\leq \eta, &\|\alpha u_1^f\|_{\ce^{10}} &\leq \eta, \\
\|\alpha^3 v_0\|_{\ce^{11}} &\leq \eta, &\|\alpha^3 v_1\|_{\ce^{10}} &\leq \eta, &
\|\alpha^3 v_0^f\|_{\ce^{11}} &\leq \eta, & \|\alpha^3 v_1^f\|_{\ce^{10}} &\leq \eta,
\end{aligned}$$
and these functions satisfy \eqref{CoupleCompatibility}. We can now apply Theorem \ref{Mainresult}, and for any support and time $T>0$ compatible with that support, we get $(u,v,h)$ with initial and final conditions $(\alpha u_0, \alpha u_1, \alpha u_0^f, \alpha u_1^f, \alpha^3 v_0, \alpha^3 v_1, \alpha^3 v_0^f, \alpha^3 v_1^f)$ such that
$$\left\{\begin{aligned}
\wave_{\nu_1} u & = h, \\
\wave_{\nu_2} v & = u^3, \\
u_{|\partial \Omega} & =  0, \\
v_{|\partial \Omega} & =  0.
\end{aligned} \right.$$
Then we also have
$$\left\{\begin{aligned}
 \wave_{\nu_1} \alpha^{-1} u & = \alpha^{-1} h, \\
\wave_{\nu_2} \alpha^{-3} v & = (\alpha^{-1} u)^3, \\
\alpha^{-1} u(0) & = \alpha^{-1} u(L) =  0, \\
\alpha^{-3} v(0)& =\alpha^{-3} v(L)=  0,
\end{aligned} \right.$$
Thus, $\alpha^{-1} h$ steers $(u_0,u_1, v_0, v_1)$ to $(u_0^f,u_1^f, v_0^f, v_1^f)$ in $T$. 

\noindent Finally, to get estimate \eqref{DegenerateEstimate}, recall \eqref{LocalEstimate} 
$$
\|h\|_{\ce^6} \leq R,
$$
hence, in terms of the original control system,
$$
\begin{aligned}
\|\alpha^{-1}h\|_{\ce^6} & \leq \alpha^{-1} R \\
& \leq \frac{2R}{\eta}\left(\|(u_0,u_1,u_0^f,u_1^f)\|_{\left(\ce^{11}\times\ce^{10}\right)^2}+\|(v_0,v_1,v_0^f,v_1^f)\|^{\frac13}_{\left( \ce^{11}\times\ce^{10}\right)^2}\right).
\end{aligned}$$
This proves Theorem \ref{GlobalResult}.


 \subsection{A general criterion for internal controllability}\label{GeneralMethod}
 
 Let us now give a general definition, which gives the main criterion our return trajectories must fulfill to apply our method:
 \begin{Def}
 A \textup{suitable return trajectory} for time $T>0$ is a trajectory $(\bar{u},\bar{v},\bar{h}) \in \ce^{11}([0,T]\times [0,L])^3$ of system \eqref{SystemGeneral}, such that 
 $$\begin{aligned}
 \bar{u}(0,\cdot)=0, \ & \bar{v}(0,\cdot)=0, \\
 \bar{u}_t(0,\cdot)=0, \ & \bar{v}_t(0,\cdot)=0, \\
 \bar{u}(T,\cdot)=0, \ & \bar{v}(T,\cdot)=0, \\
 \bar{u}_t(T,\cdot)=0, \ & \bar{v}_t(T,\cdot)=0, 
 \end{aligned}$$
 $$\textup{supp} \ (\bar{u},\bar{v},\bar{h}) \subset [0,T]\times [a,b],$$
 $$\mathscr{D}(\bar{u},\bar{v},\bar{h}) =(0,0),$$
 and such that there exists $0<\delta <\min\left(T/2,(b-a)/2\right)$ satisfying \eqref{delta}, a $\delta$-covering set $\mathcal{Q}_\delta$, a smooth closed set $\mathcal{Q}$ such that $\mathcal{Q}_\delta \subset \overset{\circ}{\mathcal{Q}}$ such that
 $$\forall (t,x) \in \mathcal{Q}, \ \bar{u}(t,x) \neq 0.$$
 \end{Def}
We can now give a general statement to sum up our work on system \eqref{System}:
 \begin{Prop}\label{Method}
 Let $0\leq a < b \leq L$, and $T>0$ such that \eqref{TimeCondTh1} holds. Suppose condition \eqref{NonDegenerateCoupling} does not hold.
 If one can find a suitable return trajectory, then system \eqref{SystemGeneral} is locally controllable in time $T$ for $\left(\ce^{11} \times \ce^{10}\right)^4$ initial and final conditions, with $\ce^{6}$ trajectories, and with a $\ce^{6}$ control supported in $[0,T]\times [a,b]$.
 \end{Prop}

  \section{Further questions}
    
 \subsection{Regularity}
 
 \
 
 Our method requires somewhat specific regularities: $\ce^{11}$ ($\ce^{10}$ for the time-derivative) for the initial and final data. As is often the case when using a Nash-Noser scheme, these regularities are probably not optimal. However, if we require for example $\ce^k$ regularity for the control, $k\geq 2$, the initial and final data have to be at least one notch smoother. Indeed, note
$$w:=u_t-\nu_1 u_x,$$
 and consider to the following computation, where one requires the control and the trajectories to be $\ce^k$:
 $$\begin{array}{rl}
 \frac{d}{dt} w(t, x-\nu_1 t)&=w_t - \nu_1 w_x \\
 &= w_t + \nu_1 w_x - 2\nu_1 w_x \\
 &= h(t, x-\nu_1 t) - 2\nu_1 (u_{tx} - \nu_1 u_{xx}) \\
 &=h(t, x-\nu_1 t) - 2\nu_1 \frac{d}{dt} u_x(t, x-\nu_1 t),
 \end{array}$$
 hence, for a fixed $t>0$ and for characteristics going from $\{0\} \times [0,L]$ to $\{t\} \times [0,L]$,
 \begin{equation}\label{Characteristics}
 \int_0^t h(s, x-\nu_1 s)ds= u_t(t,x-\nu_1 t)-u_1(x) -\nu_1 u_x(t, x-\nu_1 t) + \nu_1 u_0^\prime (x) -2\nu_1 (u_x(t, x-\nu_1 t)-u_0^\prime(x)). 
 \end{equation}
 Now the left-hand side of \eqref{Characteristics} is of class  $\ce^k$, so we need to upgrade the regularity of $u_0$ to $\ce^{k+1}$, and that of $u_1$, $x \mapsto u_x(t,x)$ and $x \mapsto u_t(t,x)$ to $\ce^k$. This shows a partial derivative loss (in the space dimension) between the trajectory and the control, and, taking $t=T$, a derivative loss between the initial and final data for $u$, and the control.
 
For linear cascade systems with smooth coefficients, the same procedure can be repeated on the second equation to establish similar losses of derivatives, showing that because $u$ has increased spatial regularity, the initial and final data for $v$ have to be $\ce^{k+2} \times \ce^{k+1}$. Note that with two controls, this would not be the case, as each control ``absorbs'' the loss of derivatives in each equation.
 
Furthermore, it would be interesting to consider other iteration schemes such as the one presented in \cite{Book}, section 4.2.1, where one considers the following linear system:
 \begin{equation}
 \left\{\begin{aligned}
 \wave u&=f_1(0,0)+g_1^v(a,b) v+g_1^u(a,0)u +h \\
 \wave v&= f_2(0,0)+g_2^v(0,b) v+g_2^u(a,b)u,\end{aligned}
 \right.
 \end{equation}
 where, for $i \in \{1,2\}$,
 $$g_i^u(u,v)=\left\{ 
 	\begin{aligned} \frac{f_i(u,v)-f_i(0,v)}{u} \ &\textrm{for} \ u \neq 0 \\
    				\partial_u f_i(0,v) \ &\textrm{for} \ u =0.\end{aligned}\right.
 $$
 $$g_i^v(u,v)=\left\{ 
 	\begin{aligned} \frac{f_i(u,v)-f_i(u,0)}{v} \ &\textrm{for} \ v \neq 0 \\
    				\partial_v f_i(u,0) \ &\textrm{for} \ v =0.\end{aligned}\right.
 $$
 Then, by superposition one can restrict to the study of
 \begin{equation}
 \left\{\begin{aligned}
 \wave u&=h \\
 \wave v&=g_2^u(a,b)u. \end{aligned}
 \right.
 \end{equation}
 But ultimately, this only shifts the problem of the $\ce^k$ regularity gap between data and control, although we now have a linear system instead of a semilinear one. 
 
 On the other hand, it is now well known that in some cases, results that were obtained using the Nash-Moser iteration scheme can also be obtained through more classical iteration schemes, see for example the works of Matthias Gunther (\cite{Gunther1}, \cite{Gunther2}). It would be interesting to know if a similar do-over is possible for our result.
 
 Finally, it would be interesting to investigate a $H^k$ version of this result, using other versions of the Nash-Moser implicit function theorem.
 \subsection{Other degenerate couplings}
 
 \
 
 Our scheme of proof also allows to prove a controllability result for systems of the form
\begin{equation}\label{System2}\left\{
\begin{aligned}
\wave_{\nu_1} u & = G(u,v)+ h, \ G \in \ce^\infty(\RR^2),\\
\wave_{\nu_2} v & = u^3, \\
u_{|\partial \Omega} & =  0, \\
v_{|\partial \Omega} & =  0.
\end{aligned} \right.\end{equation}
Indeed, this simply adds a term in the definition of $\bar{h}$ when we build our return trajectory. However, $\bar{h}$ is no longer supported in $[0,T] \times [a,b]$. The other steps remain unchanged, as the additional $G$ term does not prevent the differential operator $\mathscr{D}$ from being algebraically solvable. So we get a local internal controllability result with the same time conditions, but no condition on the support of the control. Finally, if $G$ is homogeneous of degree $1$, we can use the scaling argument to deduce a global result. 

In addition to adding a coupling term to the first equation, we can also change the power of the coupling term in the second equation. There are two cases:
\begin{enumerate}
\item \textit{Even powers} As such, our method cannot work for even powers: indeed, $u^{2k}$ has nonnegative values. In particular, by the same convexity argument as in Remark \ref{VanishingPoints}, solutions to the stationary problem cannot vanish smoothly in $1$. So the perturbative approach would allow us to build smooth return trajectories only if $u$ (and thus $h$) is spatially supported in all of $[0,L]$.

Another way of answering this question would be to switch to complex values, as is done in the appendix of \cite{CGR} for the quadratic case. 

\item \textit{Odd powers} Thanks to Proposition \ref{Method}, we know that the part that requires the most work is the construction of return trajectories: say the power of the coupling is $ 2k+1, \ k\in \NN^\ast$, in order to control all the derivatives of $v$ up to $2k+1$, we would have to look for $v$ in the form
$$\sum_{i=0}^{2k+1} f_i(t) g_i\left(\frac{|x-\frac{b+a}{2}|}{\lambda}\right).$$
This would call for ever longer computations, and for now there is no indication that there might or might not be new difficulties with these additional terms.

\end{enumerate}
\subsection{Boundary controllability}

\

In this article we have explored a method to prove internal controllability with one control. However, to our knowledge there is no result for boundary controllability with one control for semilinear systems such as \eqref{SystemGeneral}. Although boundary controllability is relatively easy to establish for simple equations, or when there are the same number of controls and equations, we cannot use results on the inversion of differential operators to reduce the number of controls.

\section*{Acknowledgements}

I would like to thank Jean-Michel Coron, who suggested that I work on this problem, for his continuous support and valuable remarks. I would also like to thank Shengquan Xiang, Amaury Hayat, Pierre Lissy and Frédéric Marbach for our discussions on this problem.  Finally, I would like to express my gratitude towards the ETH-FIM for their generous support and their hospitality.

\end{document}